\documentclass{article}

\usepackage[utf8]{inputenc} 
\usepackage[T1]{fontenc}    
\usepackage{hyperref}       
\usepackage{url}            
\usepackage{booktabs}       
\usepackage{amsfonts}       
\usepackage{nicefrac}       
\usepackage{microtype}      
\usepackage{lipsum}
\usepackage{lmodern}
\usepackage[all]{hypcap}    

\usepackage{amsmath}
\usepackage{amssymb}
\usepackage{amsthm}
\usepackage{geometry}
\usepackage{enumerate}
\usepackage{tikz-cd}
\usepackage{tikz}
\usetikzlibrary{positioning}
\usepackage{mathtools}
\usepackage{multicol}
\usepackage{tabularx} 
\usepackage{mathtools}
\usepackage{newtxmath}

\newtheorem{theo}{Theorem}[section]
\newtheorem{prop}[theo]{Proposition}

\newtheorem{lem}[theo]{Lemma}
\newtheorem{cor}{Corollary}[theo]

\newtheorem{rem}[theo]{Remark}

\title{Computation of the Nielsen fixed point number for $2$-valued non-split maps on the Klein bottle}

\author{Bartira Maués\\
	\texttt{Departamento de Matemática - UFSCar}
}

\begin{document}
	\maketitle

	\begin{abstract} In this paper we study $2$-valued non-split maps, focusing on the Klein bottle.
		We establish a connection between a $2$-valued non-split map $\phi:X\multimap Y$ and a pair of classes of maps $([f],[f\circ \delta])\in [\tilde X,Y]\times[\tilde X, Y]$, where $\delta$ is a free involution on $\tilde X$, $X=\tilde X/\delta$ and the class of the lift factor $[f]$ does not satisfy the Borsuk-Ulam Property in respect to $\delta$.
		We also exhibit a method to compute the Nielsen fixed point number of a $2$-valued non-split map on a closed connected manifold in terms of the Nielsen coincidence number between a lift factor and a covering space map, generalizing the formula from only orientable manifolds to also non-orientable manifolds.
		Finally we display a formula for the Nielsen fixed point number of $2$-valued non-split maps on the Klein bottle in terms of two braids of the Klein bottle.
	\end{abstract}


\section{Introduction}

	The Nielsen fixed point number is a topological invariant that provides a way to determine the existence of fixed points of maps on some topological spaces.
	Every self-map $f$ on these spaces has at least $N(f)$ fixed points, where $N(f)$ is the Nielsen fixed point number of $f$, hence we have a lower bound on the number of fixed points, even when finding all fixed points is difficult.
	Efforts have been made to compute a formula for the Nielsen-fixed point number on different spaces.
	For compact closed surfaces with non-negative Euler characteristis an algebraic formula is well-known.

	In 1984 H. Schirmer starts exploring the Nielsen theory regarding $n$-valued maps on a closed manifold (see \cite{schirmer1984fix, schirmer1984index, schirmer1985minimum}) and exhibits a way to compute the Nielsen number of $n$-valued split maps in terms of the Nielsen number of certain single-valued maps.
	Following the works of H. Schirmer, numerous studies focusing on Nielsen theory of $n$-valued maps emerge, particularly in recent years.
	
	An algebraic formula for the Nielsen fixed-point number of $2$-valued split maps on the Torus (\cite{gonccalves2018fixed}) and of $n$-valued split maps on the Klein bottle (\cite{gonccalves2025nielsen}) are established.
	For $2$-valued non-split maps on orientable closed compact manifolds a method to compute the Nielsen number is described in \cite{gonccalves2017fixed}, which can be applied relatively straightforward on the Torus.
	This method holds also for certain $n$-valued non-split maps on orientable closed compact manifolds.
	There are various studys for Nielsen theory on $n$-valued maps on closed compact surfaces, although numerous questions continue to linger for this scenario, particularly for non-split maps on non-orientable surfaces, which remain relatively underexplored. 
	The primary objective of the current work is to study the computation of the Nielsen number of $2$-valued non-split maps on non-orientable closed compact manifolds, particularly on the Klein bottle.
	
	This work is divided in four sections, besides the introduction and the final section, where some open questions and possible next steps are displayed.
	In Section \ref{sec:2valued maps &BU} we exhibit some preliminaries on $2$-valued maps and their lift-factors and, subsequently, prove the following theorem, which connects $2$-valued maps to the Borsuk-Ulam Property of a lift-factor, giving us some restrictions on $2$-valued maps, simplifying the computation of the Nielsen number.\\	
	
	\noindent{\bf Theorem~ \ref{Phi map=>-f_1 BU}.}
		{\it
			Let $\phi:X\multimap Y$ be a $2$-valued non-split map with correspondent map $\Phi:X\rightarrow D_2(Y)$, let $q:\tilde X \rightarrow X$ be the covering space induced by $\Phi_\#^{-1}(P_2(Y))$ and $\hat \Phi=\{f_1,f_2\}:\tilde X\rightarrow F_2(Y)$ the lift of $q\circ\Phi$.
			Let furthermore $\delta$ be the unique non-trivial deck transformation.
			Then $[f_1]$ does not have the Borsuk-Ulam Property in respect to $\delta$ and it holds that $f_{2}=f_{1}\circ \delta$.
		}\\
		
	In Section \ref{sec:Nielsen nr of 2val non-split} we show the following theorem, in which we exhibit a way to compute the Nielsen number of $2$-valued non-split map on closed connected manifold, extending the existing result to non-orientable manifolds as well.\\
	
	\noindent{\bf Theorem~ \ref{N(phi)=N(q,f1)}.}
		{\it
			Let $X$ be a closed connected manifold and let $\phi:X\multimap X$ be a $2$-valued non-split map with covering space $q:\tilde X\rightarrow X$ induced by the subgroup $\Phi^{-1}_\#(P_2(X))$ and let $\hat \Phi=\{f_1, f_2\}$ be a lift of $\Phi\circ q$. Then the Nielsen fixed point number of $\phi$ is equal to the Nielsen coincidence number of the pair $(q, f_1)$, respectively the pair $(q, f_2)$, i.e.
				$N(\phi)=N(q, f_1)=N(q, f_2).$
		}\\
	
	Subsequently we focus on the Klein bottle and give, in Section \ref{braidgroups}, the presentation of the pure and total braid groups on two strands on the Klein bottle, that we use in this paper.
	Finally, in Section \ref{sec:Nielsennumber}, we study $2$-valued non-split maps on the Klein bottle and their lift factors, dividing in two cases, depending whether the induced covering space is the Torus or the Klein bottle
	and prove the following explicit formula for $2$-valued non-split maps on the Klein bottle in terms of two braids.\\
	
	\noindent{\bf Theorem~ \ref{Nielsennumber non-split}.}
		{\it
		Let $\phi:\mathbb K\multimap \mathbb K$ be a $2$-valued non-split map on the Klein bottle $\mathbb K$, where $\pi_1(\mathbb K)=\langle \alpha,\beta:\alpha\beta\alpha=\beta\rangle$ and let $\Phi:\mathbb K\rightarrow D_2(\mathbb K)$ be the correspondent map.
		Let the induced homomorphism be written as
		\begin{align}
				\Phi_\#: 
			\alpha &\longmapsto w_1(a_2,b_2)a_1^{r_1}b_1^{s_1}\sigma^{k_1}\label{Phi nonsplit}\\
			\beta &\longmapsto w_2(a_2,b_2)a_1^{r_2}b_1^{s_2}\sigma^{k_2},\nonumber
		\end{align}
		where  
		$k_1,k_2\in \{0,1\}$, $r_1,r_2,s_1,s_2\in \mathbb N$ and $pr(w_i)=a_2^{m_i}b_2^{n_i}$, for $i=0,1$, where $pr:F_2(a_2,b_2)\rightarrow  F_2(a_2,b_2)/\langle a_2b_2a_2b_2^{-1}\rangle$ is the projection as in Remark \ref{elements in B2(K)}\eqref{pr}.
		Then the Nielsen fixed point number of $\phi$ is given as
		\begin{align*}
			N(\phi)=
			\begin{cases}
				|1-s_2|\cdot max\{|(1+(-1)^{ s_1}) r_1+m_1|,2\}& \text{ if } k_1=1, k_2=0, \\  
				|1-s_2+s_1|\cdot max\{|(1+(-1)^{ s_1}) r_1+m_1|,2\}& \text{ if } k_1=1, k_2=1, \\
				|r_1|\cdot|2s_2+n_2-2|& \text{ if } k_1=0, s_1,n_2 \text{ even and } r_1\neq 0, \\
				|2s_2+n_2-2| & \text{ otherwise.}
			\end{cases}
		\end{align*}}	
	{\bf Acknowledgement:} Some of these results were obtained during the author's Ph.D, in which they received financial support by CAPES and CNPq (S\~ ao Paulo -- Brazil).

\section{Preliminairies on 2-valued maps and a connection to the Borsuk-Ulam Property}\label{sec:2valued maps &BU}
	We follow the path conducted by D. L. Gon\c calves and J. Guaschi, in the works \cite{gonccalves2017fixed} and \cite{gonccalves2018fixed}, where a methodology is formulated to explore Nielsen fixed point theory of $n$-valued maps utilizing braid theory.
	In this paper we assume that every topological space is path-connected, locally path-connected
	and semilocally simply-connected Hausdorff space. 
	Let $X$ and $Y$ be such spaces, then an \textit{$n$-valued map} from $X$ to $Y$, denoted by $\phi:X\multimap Y$, is a multimap, specifically a (continuous) correspondence $\phi$ that assigns for each $x \in X$ a subset of exactly $n$ points $\{y_1, \dots, y_n\} \subseteq Y$.
	We study $2$-valued maps, that is, we set $n=2$ from here on out, but a lot of the following properties can be generalized for $n$-valued maps as well and the results of this section and the next can be generalized for $n$-valued to some degree.
	
	Let us consider the \textit{ordered configuration space} of $Y$, given as $F_2(Y) = \{(y_1, y_2)\in Y^2 : y_1 \neq y_2\}$.
	Next consider the  $2$-th symmetric group $\mathcal{S}_2=\{id,(12)\}$.
	Recall the action of $\mathcal S_2$ on the ordered configuration space $F_2(Y)$ defined by $( \tau, (y_1, y_2)) \mapsto  \tau\cdot (y_1,  y_2)=(y_{ \tau(1)}, y_{\tau(2)})$, for any $\tau\in \mathcal S_2$.
	By taking the quotient of $F_2(Y)$ under this action by $\mathcal{S}_2$, we obtain the \textit{unordered configuration space} $D_2(Y) = F_2(Y)/\mathcal{S}_2$. 
	Note that there exists a natural bijection between the subsets of a topological space $Y$ of cardinality $2$ and the unordered configuration space $D_2(Y)$.
	This yields a bijective correspondence between $2$-valued maps from $X$ to $Y$ and maps from $X$ to $D_2(Y)$.
	From this point forward, we will denote by $\Phi$ the function from $X$ to $D_2(Y)$ that corresponds to the $2$-valued map $\phi: X \multimap Y$. 
	This correspondence preserves continuity under relatively mild conditions imposed on $X$ (see \cite[Corollary 4.1]{brown2018topology}), and by extension, homotopy as well.
	It is generally known that $\pi_1(D_2(Y))$ is isomorphic to the braid group $B_2(Y)$ and $\pi_1(F_2(Y))$ to the pure braid group $P_2(Y)$.
	
	We say that a $2$-valued map $\phi:X \multimap Y$ \textit{splits} if there exist (continuous) functions $f_1, f_2: X \rightarrow Y$ such that $\phi(x) = \{f_1(x), f_2(x)\}$ holds for all $x \in X$, and we denote this as \textit{$\phi = \{f_1, f_2\}$}.
	According to \cite[page 8]{gonccalves2018fixed}, a $2$-valued map $\phi:X \multimap Y$ splits if and only if the function $\Phi$ possesses a \textit{lift}, meaning that there exists a map $\hat{\Phi}=\{f_1, f_2\}: X \rightarrow F_2(Y)$ such that $\Phi=\pi \circ \hat \Phi$, which exists if and only if, $\Phi(\pi_1(X))\subseteq P_2(Y)$ (see \cite[Proposition 1.33.]{hatcher2002algebraic}).
	We call the single-valued maps $f_1, f_2$ the \textit{lift factors of $\Phi$, respectively $\phi$}.
	
	For non-split maps we follow the approach of R. Brown and D. Gon\c calves in \cite{brown2023lift}.
	Let $\phi:X\multimap Y$ be a $2$-valued non-split map and its correspondent map $\Phi:X\rightarrow D_2(Y)$, then $\Phi(\pi_1(X))\subsetneqq P_2(Y)$ and consequently the group $\Phi_\#^{-1}(P_2(Y))$ is a proper subgroup of $\pi_1(X)$.
	By \cite[Proposition 1.36]{hatcher2002algebraic} there is a $2$-sheeted covering space $q:\tilde X\rightarrow X$ such that $q_\#(\pi_1(\tilde X))=\Phi_\#^{-1}(P_2(Y))$
	and since $(\Phi\circ q)_\#(\pi_1(\tilde X))\subseteq P_2(Y)$, there is a lift $\hat \Phi=\{f_1,f_2\}:\tilde X\rightarrow F_2(Y)$ of $\Phi\circ q$.
	\begin{multicols}{2}
		\noindent
		\begin{equation*}
			\begin{tikzcd}[column sep=huge]
				&&F_2(Y)\arrow{d}{\pi}\\
				\tilde X \arrow{r}[swap]{q}\arrow[dashed]{urr}{\hat \Phi=\{f_1, f_2\}} &X\arrow{r} [swap]{\Phi}&D_2(Y).
			\end{tikzcd}
		\end{equation*}
		
		\noindent
		\begin{equation*}
			\begin{tikzcd}[column sep=huge]
				\pi_1(\tilde X) \arrow[dashed]{r}{\hat\Phi_\#}\arrow{d}[swap]{q_\#}&P_2(X)\arrow{d}{\pi_\#}\\
				\pi_1(X) \arrow{r}[swap]{\Phi_\#}&B_2(X)
			\end{tikzcd}
		\end{equation*}
	\end{multicols}

	For the next step we explore $f_1$ and $f_2$, the so-called \textit{lift factors of $\Phi$, respectively $\phi$} and analyze them with the Borsuk-Ulam Property in mind. 
	To recall the definition of a generalization of the Borsuk-Ulam Property, let $\tilde X$ and $Y$ be topological spaces and let $\delta:\tilde X\rightarrow \tilde X$ be a free involution, i.e. $\delta$ is fixed point free and $\delta^2=id$.
	We say that a homotopy class $[f_0]\in [\tilde X,Y]$ has the \textit{Borsuk-Ulam Property with respect to $\delta$} if, for every $f\in [f_0]$, there exists an $\tilde x\in \tilde X$ such that $f(\tilde x)=(f\circ\delta)(\tilde x)$.
	
	\begin{rem}\label{fBU->2valued non split map}
		Observe that if $[f]\in [\tilde X,Y]$ is a homotopy class that does not have the Borsuk-Ulam Property in respect to a free involution $\delta:\tilde X\rightarrow \tilde X$,		
		then there is a map $f_1\in [f]$ such that $f_1$ and $f_1\circ \delta$ do not coincide.
		Consider the map $\hat \Phi=(f_1,f_1\circ\delta):\tilde X\rightarrow F_2(Y)$ and the quotient space $X=\tilde X/ \delta$, which together with the projection $p:\tilde X\rightarrow X$ is a (normal) covering space.
		Then $\hat \Phi$ induces a map $\Phi:X\rightarrow D_2(Y)$ on the quotient spaces and it holds that $\hat\Phi$ is a lift of $p\circ\Phi$.
		Let $\phi:X\multimap Y$ be the correspondent $2$-valued non-split map, then we found a $2$-valued non-split map where a lift factor is homotopic to $f$.
	\end{rem}
	The next theorem is the converse of this statement and was shown in \cite[Proposition 2.2.5.]{maues2020nielsen}.
	\begin{theo}\label{Phi map=>-f_1 BU}\label{fsigmak=fdeltak}
		Let $\phi:X\multimap Y$ be a $2$-valued non-split map with correspondent map $\Phi:X\rightarrow D_2(Y)$, let $q:\tilde X \rightarrow X$ be the covering space induced by $\Phi_\#^{-1}(P_2(Y))$ and $\hat \Phi:\tilde X\rightarrow F_2(Y)$ a lift of $q\circ\Phi$.
		Let furthermore $\delta$ be the unique non-trivial deck transformation.
		Then $[f_1]$ does not have the Borsuk-Ulam Property in respect to $\delta$ and it holds that $f_{2}=f_{1}\circ \delta$.
	\end{theo}
	\begin{proof}
		By construction $\hat \Phi$ is equivariant, i.e. the following diagram commutes
		\begin{equation*}
			\begin{tikzcd}[column sep=huge]
				\tilde X \arrow{d}[swap]{\delta}\arrow{r}{\hat \Phi=\{f_1,f_2\}}&F_2(Y)\arrow{d}{(12)}\\
				\tilde X \arrow{r}[swap]{\hat \Phi=\{f_1,f_2\}} & F_2(Y).
			\end{tikzcd}
		\end{equation*}
		Therefore it holds, for every $\tilde x\in \tilde X$, that
		\begin{equation*}
			(f_1(\delta(\tilde x)), f_2(\delta(\tilde x)))=(\hat\Phi\circ\delta)(\tilde x)=(12) \cdot \hat \Phi(\tilde x)= (12)\cdot(f_1(\tilde x),f_2(\tilde x))=(f_2(\tilde x),f_1(\tilde x)).
		\end{equation*}		
		It follows that	$f_2(\tilde x)=(f_1\circ\delta)(\tilde x)$, for every $\tilde x\in \tilde X$ and consequently $f_2=f_1\circ \delta$.
		Note that $\delta$, the unique non-trivial deck transformation $q:\tilde X\rightarrow X$ is a free involution.
		Since $\hat \Phi=\{f_1, f_2\}$ is a map, $f_1$ and $f_2=f_1\circ \delta$ have no coincidence points, i.e. $f_1(\tilde x)\neq f_1\circ \delta(\tilde x)$, for any $\tilde x\in \tilde X$.
		Therefore $[f_1]$ does not have the Borsuk-Ulam Property in respect to $\delta$.
	\end{proof}

\section{Nielsen number for 2-valued non-split maps on connected closed manifolds}\label{sec:Nielsen nr of 2val non-split}
	To compute the Nielsen number of a $2$-valued non-split map $\phi$ on a connected, closed orientable manifold $X$ we can use the formula, given in \cite[Corollary 23]{gonccalves2017fixed}.
	To generalize for connected, closed manifolds and therefore include the non-orientable manifolds,
	we look for maps that "preserve" the local orientation on the non-orientable manifolds, the so called orientation-true maps.
	Let $\tilde X$ and $X$ be connected manifolds, we say a map $g:\tilde X\rightarrow X$ is orientation-true, if given any loop $\gamma:I\rightarrow \tilde X$ the loop $g\circ \gamma:I\rightarrow X$ preserves (reverses) orientation if and only if $\gamma$ preserves (reverses) orientation (see \cite[page 406]{brown2005handbook}).
	As an example consider the Möbius strip $M$ and a map $f$ on $M$, then $f$ is orientation-true if it takes loops that traverse
	a Möbius strip once to loops that traverse a Möbius strip once and loops that don't to loops that don't.
	Consider two maps $f, g: \tilde X\rightarrow X$ with $g$ orientation-true, then the coincidence index for the pair $(g,f)$ of an isolated coincidence point is an integer defined as the local index of the pair $(g, f)$ (see \cite[Definition 5.1.]{gonccalves1997lefschetz}).
	The map $g$ being orientation-true guarantees that there is a consistent global choice of local orientations.
	Note that every covering map $q:\tilde X\rightarrow X$ is orientation-true, see for example \cite[Proposi\c c\~ao 1.1.5]{taneda2007teoria}.
		
	In the next proposition we show and use, that there is a bijection between $Fix(\phi)$, the set of fixed points of a $2$-valued non-split map $\phi$, and $Coin(q,f_i)$, the set of Coincidence points of the covering map $q$ and a lift factor $f_i$ of $\phi$.
	Furthermore this bijection takes Nielsen coincidence classes to Nielsen fixed point classes, which allows us to compare the Nielsen fixed point number of $\phi$ with the Nielsen coincidence number of the pair $(q,f_i)$.
	
	\begin{theo}\label{N(phi)=N(q,f1)}
		Let $X$ be a closed connected manifold and let $\phi:X\multimap X$ be a $2$-valued non-split map with covering space $q:\tilde X\rightarrow X$ induced by the subgroup $\Phi^{-1}_\#(P_2(X))$ and let $\hat \Phi=\{f_1, f_2\}$ be a lift of $\Phi\circ q$. Then the Nielsen fixed point number of $\phi$ is equal to the Nielsen coincidence number of the pair $(q, f_1)$, respectively the pair $(q, f_2)$, i.e.
		$N(\phi)=N(q, f_1)=N(q, f_2).$
	\end{theo}
	\begin{proof}
		We start by showing that the covering map $q:\tilde X\rightarrow X$ induces a bijection $q_{|Coin_i}$ between $Coin(q, f_i)$ and $Fix(\phi)$, for both $i=1$ and $i=2$.
		This map is well-defined, since for any $\tilde x\in Coin(q, f_i)$ we have that $q(\tilde x)=f_i(\tilde x)$ is a fixed point of $\phi$.
		To show that it is a bijection, let $x_0$ be a fixed point of $\phi$, i.e. $x_0\in Fix(\phi)$.
		Since $q:\tilde X\rightarrow X$ is a $2$-sheeted covering, the fiber has exactly two (distinct) points, i.e. $q^{-1}(x_0)=\{\tilde x_1,\tilde x_2\}$.
		Then $x_0$ is one of the two points of the set $\Phi(x_0)=(\Phi\circ q)(\tilde x_1)$.
		Since $\hat \Phi$ is the lift of $\Phi\circ q$ it follows that $x_0\in \hat \Phi(\tilde x_1)=\{f_1(\tilde x_1), f_2(\tilde x_1)\}$.
		Therefore either $f_1(\tilde x_1)=x_0$, or $f_2(\tilde x_1)=x_0$ and since $q(\tilde x_1)=x_0$, it follows that either $\tilde x_1\in Coin(q,f_1)$ or $\tilde x_1\in Coin(q,f_2)$.
		Consider the unique non-trivial deck transformation $\delta:\tilde X\rightarrow \tilde X$.
		Since $\tilde x_1$ and $\tilde x_2$ are the only two points in the fiber of $x_0$ it holds that $\delta$ takes $\tilde x_1$ to $\tilde x_2$ and vice-versa.
		Therefore, if $\tilde x_1\in Coin(q,f_1)$, i.e. $f_1(\tilde x_1)=x_0$, then $f_2(\tilde x_2)=x_0$, i.e. $\tilde x_2\in Coin(q,f_2)$.
		Analogously, if $\tilde x_2\in Coin(q,f_1)$, then $\tilde x_1\in Coin(q,f_2)$.
		Therefore $Coin(q, f_1)\cap q^{-1}(x_0)$ and $Coin(q, f_2)\cap q^{-1}(x_0)$ are disjoint sets, both containing exactly one point and their union is $\{\tilde x_1, \tilde x_2\}$.
		Hence $q$ restricts to a bijection $q_{|Coin_i}:Coin(q, f_i)\rightarrow Fix(\phi)$ for both $i=1$ and $i=2$.
		
		This bijection $q_{|Coin_1}$ (respectively $q_{|Coin_2}$) maps, by Lemma 18 of \cite{gonccalves2017fixed}, a Nielsen coincidence class $[\tilde x_1]$ of the pair $(q,f_1)$ (respectively of the pair $(q,f_2)$) bijectively to a Nielsen fixed point class $[q(\tilde x_1)]$ of $\phi$.
		We show next that the coincidence index of such a class $[\tilde x_1]$ for the pair $(q,f_i)$ is equal to the Nielsen fixed point index of the class $[q(\tilde x_1)]$.
		For that reason, let $i\in \{1,2\}$, let $x_0\in Fix(\phi)$ be an isolated fixed point of $\phi$ and let $\{\tilde x_1\}= q^{-1}(x)\cap Coin(q, f_i)$.
		Then there is a contractible neighborhood $U\subset X$ of $x_0$ such that $Fix(\phi)\cap U=\{x_0\}$ and $q^{-1}(U)=\tilde U \cup \tilde V$, where $\tilde U$ and $\tilde V$ are disjoint open subsets of $\tilde X$ homeomorphic to $U$ by $q$.
		Suppose without loss of generality that $\tilde x_1\in \tilde U$.
		The neighborhood $U$ satisfies the conditions of the Splitting Lemma \cite[Lemma 2.1.]{schirmer1984index} and therefore
		there exist two maps $g_1,g_2:U\rightarrow X$ such that $\phi|_U=\{g_1,g_2\}$.
		Since $\hat \Phi=\{ f_1, f_2\}:\tilde X\rightarrow F_2(X)$ is the lifting of $\Phi\circ q:\tilde X\rightarrow F_2(X)$ we have for every $y\in U$ that
		$\{ f_1(y), f_2(y)\}=\{(g_1\circ q)(y),(g_2\circ q)(y)\}.$
		Suppose without loss of generality that $( f_i)|_{\tilde U}=g_i\circ q_{|\tilde U}$.
		As $q$ restricts to a homeomorphism to $\tilde U$ it holds that $g_i= f_i\circ (q_{|\tilde U})^{-1}$.
		Since $\tilde x_1\in q^{-1}(x)\cap Coin(q, f_i)$ it follows that $x \in Fix(g_i)$.
		By definition of the fixed point index 
		\begin{equation*}
			ind(\phi, U)=ind(g_i,U)=ind(f_i\circ (q|_{\tilde U})^{-1}, U),
		\end{equation*}
		which is exactly the local coincidence index of the pair $(q, f_i)$ in $\tilde U$, where the local orientation in $\tilde U$ is determined by the local homeomorphism $q$.
		Since $q: \tilde X\rightarrow X$ is a covering map and therefore orientation-true (see for example \cite[Proposi\c c\~ao 1.1.5]{taneda2007teoria}), the local index extends globally, that is the index of the class is the sum of the indices of all coincidence points contained in the class.
		Therefore the essential coincidence classes of the pair $(q, f_i)$ correspond to the essential fixed point classes of $\phi$, implying that the Nielsen coincidence number $N(q, f_i)$
		and the Nielsen number $N(\phi)$ are equal.
	\end{proof}
	
	We expect it to be possible to give a formula for the Nielsen number of an $n$-valued map $\phi$, where
	$n$ is greater than $2$, in terms of Nielsen coincidence Theory, generalizing Theorem 6 of \cite{gonccalves2017fixed}.


\section{The braid group on 2 strands on the Klein bottle}\label{braidgroups}
	In this section we give a presentation of the pure and total braid groups on $2$ strands of the Klein bottle, which we use in the remainder of this paper.
	To start studying the pure braid groups consider a presentation given by  J. Guaschi and C. M. e Pereiro, where we already set $n=2$.
	
	\begin{prop}\cite[Theorem 2.1]{guaschi2020lower} \label{presentacao Puras Carolina(prop)}
		The following is a presentation of the pure braid group $P_2(\mathbb K)$ of the Klein bottle $\mathbb K$, where the 
		generating set is $\{\hat a_1,\hat b_1, \hat a_2, \hat b_2, \hat c_{1,2}\}$ and the relations are as follows;
		\begin{multicols}{2}
			
			\noindent
			\begin{enumerate}
				\item $\hat a_1\hat a_2\hat a_1^{-1}=\hat a_2$;
				\item $\hat a_1^{-1}\hat b_2\hat a_1=\hat b_2\hat a_2\hat c_{1,2}^{-1}\hat a_2^{-1}$;
				\item $\hat a_1^{-1}\hat c_{1,2}\hat a_1=\hat a_2 \hat c_{1,2}\hat a_2^{-1}$;
				\item does not exist for $n=2$;
				
				\item $\hat c_{1,2}=\hat b_1 \hat a_1^{-1}\hat b_1^{-1}\hat a_1^{-1}=\hat b_2^{-1}\hat a_2\hat b_2\hat a_2$;
				\item $\hat b_1^{-1}\hat b_2\hat b_1=\hat b_2 \hat c_{1,2}$;
				\item $\hat b_1^{-1}\hat a_2\hat b_1=\hat a_2\hat b_2\hat c_{1,2}^{-1}\hat b_2^{-1}$;
				\item $\hat b_1^{-1}\hat c_{1,2}\hat b_1=\hat b_2 c_{1,2}^{-1}\hat b_2^{-1}.$
			\end{enumerate}
		\end{multicols}
	\end{prop}
	
	Note that we can write the generator $\hat c_{1,2}$ as a product of the remaining generators and can therefore remove it from the generating list.
	The pure braid group $P_2(\mathbb K)$ is isomorphic to $F_2(a_2,b_2)\rtimes_\theta\pi_1(\mathbb K)$ a semi-direct product, described in the following proposition.
	
	\begin{prop}	\cite[Lemma 3.6 and Corollary 3.7.1]{gonccalves2025nielsen}\label{P_2=produto semidireto e presentacao F2 |x pi1K}
		The pure braid group $P_2(\mathbb K)$ is isomorphic to $F_2(a_2,b_2)\rtimes_{\theta}\pi_1(\mathbb K)$, with $\pi_1(\mathbb K)=\langle a_1,b_1:a_1b_1a_1=b_1\rangle$ and action $\theta: \pi_1(\mathbb K)\rightarrow Aut(F_2(a_2,b_2))$, that maps an element of the fundamental group of the Klein bottle $a_1^rb_1^{2s}$, respectively $a_1^rb_1^{2s+1}$ to
		\begin{multicols}{2}
			\noindent
			\begin{align*}
				\theta_{a_1^rb_1^{2s}}:
				a_2&\longmapsto a_2,\\
				b_2&\longmapsto a_2^{2r}b_2.
			\end{align*}
			
			\noindent
			\begin{align*}
				\theta_{a_1^rb_1^{2s+1}}:
				a_2&\longmapsto a_2^{-1},\\
				b_2&\longmapsto a_2^{2r+1}b_2a_2^.
			\end{align*}
		\end{multicols}
		\noindent
		The isomorphism $\Theta_P:F_2(a_2,b_2)\rtimes_{\theta}\pi_1(\mathbb K)\rightarrow P_2(\mathbb K)$ is given by
		\begin{multicols}{4}
			\noindent
			\begin{align*}
				a_1 & \longmapsto \hat a_2 \hat a_1,
			\end{align*}
			
			\noindent
			\begin{align*}
				b_1 & \longmapsto \hat b_2 \hat b_1,
			\end{align*}
			
			\noindent
			\begin{align*}
				a_2 & \longmapsto \hat a_2,
			\end{align*}
			
			\noindent
			\begin{align*}
				b_2 & \longmapsto \hat b_2.
			\end{align*}
		\end{multicols}
	\end{prop}
	
	\begin{rem}\label{rem P2 normal form pri}
		\begin{enumerate}[i)]
			\item \label{F2pi1K:normal form} 		
				We describe an element in $F_2(a_2,b_2)\rtimes_\theta \pi_1(\mathbb K)$ as written in its normal form, if it is written in the form $w(a_2,b_2) a_1^sb_1^t$, where $w(a_2,b_2)$ is a word in $F_2(a_2,b_2)$ and $s,t$ are integers, i.e. $a_1^sb_1^t$ is an element in $\pi_1(\mathbb K)$.
				Two braid in $F_2(a_2,b_2)\rtimes_\theta \pi_1(\mathbb K)$ in their normal form $w_1(a_2,b_2) a_1^{s_1}b_1^{t_1}$ and $w_2(a_2,b_2) a_1^{s_2}b_1^{t_2}$ are equal, if and only if, $w_1(a_2,b_2)=w_2(a_2,b_2)$, $s_1=s_2$ and $t_1=t_2$.
			\item \label{pr1,pr2}
				To compute the lift factors $f_1$ and $f_2$ of a a non-split map we use the projections $pr_i:F_2(\mathbb K)\rightarrow \mathbb K$ to each component.
				The induced homomorphisms $(pr_1)_\#:P_2(\mathbb K)\rightarrow \langle a_1,b_1: b_1^{-1}a_1b_1a_1\rangle$ takes $\hat a_1$ to $a_1$, $\hat b_1$ to $b_1$ and both $\hat a_2,\hat b_2$ to $1$ and the homomorphism $(pr_1)_\#:P_2(\mathbb K)\rightarrow \langle a_2,b_2:b_2^{-1} a_2b_2a_2\rangle$ takes both $\hat a_1, \hat b_1$ to $1$, $\hat a_2$ to $a_2$ and $\hat b_2$ to $b_2$.
		\end{enumerate}
	\end{rem}
	
	\subsection{Total Braid Group}
	
	For the total braid group of the Klein bottle on two strands we use again a presentation given in \cite{guaschi2020lower}, where we start counting at $5)$, since the relation $1)$ to $4)$ are trivial for $n=2$.
	
	\begin{prop}\cite[Theorem 2.2]{guaschi2020lower}\label{presentation Carolina B_n(K)(teo)} 
		The $2$-th total braid group of the Klein bottle admits the following presentation.
		The generating set is $\{a,b,\sigma_1\}$, and
		the defining relations are
		\begin{multicols}{2}
			
			\noindent
			\begin{enumerate}
				\item[5.] $b^{-1}\sigma_1 a=\sigma_1 a\sigma_1 b^{-1}\sigma_1$;
				\item[6.] $a(\sigma_1 a \sigma_1)=(\sigma_1 a \sigma_1)a$;
				
				\item[7.] $b(\sigma_1^{-1} b \sigma_1)=(\sigma_1^{-1} b \sigma_1^{-1})b$;
				\item[8.] $\sigma_1^2=ba^{-1}b^{-1}a^{-1}$
			\end{enumerate}
		\end{multicols}
	\end{prop}
	
	The generators of the subgroup $P_2(\mathbb K)$, which are $\hat a_1,\hat b_1,\hat a_2$ and $\hat b_2$, are given in this presentation for $B_2(\mathbb K)$ by $a$, $b$, $\sigma_1 a \sigma_1$ and $\sigma_1^{-1} b \sigma_1^{-1}$, respectively, describes the inclusion 
	$i:P_2(\mathbb K)\rightarrow B_2(\mathbb K)$.
	Consider the map $\rho:B_2(\mathbb K)\rightarrow \mathcal S_2$, which takes each braid to its underlying permutation.
	In this case, for braids on two strands, $\rho$ takes every non-pure braid to $(12)$ and all pure braids to the identity.
	Hence we have the following short exact sequence, where we can identify $P_2(\mathbb K)$ with $F_2(a_2,b_2)\rtimes_\theta \pi_1(\mathbb K)$ through the isomorphism $\Theta_P$ given in Proposition \ref{P_2=produto semidireto e presentacao F2 |x pi1K};
	
	\begin{equation}
		\begin{tikzcd}
			1\arrow{r} & F_2\rtimes_{\theta}\pi_1(\mathbb K)
			\arrow{r}{i\circ\Theta_P} & B_{2}(\mathbb K)\arrow{r}{\rho}
			& \mathcal S_2\arrow{r}& 1, \label{sequence total braid group}
		\end{tikzcd}
	\end{equation}
	
	Using the standard method to obtain a presentation of a group extension we obtain a presentation for $B_2(\mathbb \mathbb K)$, the total braid group on two strands in terms of elements in the free group on two letters $F_2(a_2,b_2)$, the Klein bottle $\pi_1(\mathbb K)=\langle a_1,b_1:b_1^{-1}a_1b_1a_1\rangle$ and an element $\sigma$, which can be identified with the non-trivial element in $\mathcal S_2$.
		
	\begin{prop}\label{cap-braids:Presentation B_2(minha)(theo)}
		The braid group of the Klein bottle with two strings $B_2(\mathbb K)$ admits the following presentation.
		The generating set is $\{a_1, b_1, a_2, b_2,\sigma\}$
		and the defining relations are
		
		\begin{multicols}{2}
			
			\noindent
			\begin{enumerate}
				\item $ a_1 a_2= a_2 a_1$;
				\item $ a_1 b_2 a_1^{-1}= a_2^2 b_2$;
				\item $ b_1 a_2 b_1^{-1}= a_2^{-1}$;
				\item $ b_1 b_2 b_1^{-1}= a_2 b_2 a_2$;
				\item $ a_1 b_1 a_1 b_1^{-1}=1$;
				
				\item $\sigma^2= b_2^{-1} a_2 b_2 a_2$;
				\item $\sigma^{-1} a_1\sigma= a_1$;
				\item $\sigma^{-1} b_1\sigma
					=\sigma^{-2}b_1$;
				\item $\sigma^{-1} a_2\sigma
					=a_2^{-1}\sigma^2a_1$;
				\item $\sigma^{-1} b_2\sigma
					=\sigma^{-2}b_2^{-1}b_1;$
			\end{enumerate}
		\end{multicols}
	\end{prop}
	\begin{proof}
		Consider the short exact sequence \eqref{sequence total braid group}, then $(i\circ\Theta_P): F_2\rtimes_{\theta}\pi_1(\mathbb K)\rightarrow B_2(\mathbb K)$ takes $a_1,b_1,a_2$ and $b_2$ to $\sigma_1 a \sigma_1 a, \sigma_1^{-1} b \sigma_1^{-1} b, \sigma_1 a \sigma_1	$ and $\sigma_1^{-1} b \sigma_1^{-1}$ respectively.		
		From the standard method to obtain a presentation of a group extension (see for example \cite[page 139]{johnson1997presentations}) it follows that the generating set for $B_2(K)$ is $\{(i\circ\Theta_P)(a_1),(i\circ\Theta_P)(b_1),(i\circ\Theta_P)(a_2),(i\circ\Theta_P)(b_2), \sigma\}$, where $\sigma$ is a transversal of $\rho$.
		Since $\rho(\sigma)=(12)$ we can choose $\sigma$ being $\sigma_1$ of the presentation given in Proposition \ref{presentation Carolina B_n(K)(teo)}.
		Furthermore there are three types of relations, coming from the relations in $F_2\rtimes_{\theta}\pi_1(\mathbb K)$, the relation in $\mathcal S_2$ and the conjugates of generators of $F_2\rtimes_{\theta}\pi_1(\mathbb K)$ by the generators of $\mathcal S_2$.
		The first five relations are given by the relations in $F_2\rtimes_{\theta}\pi_1(\mathbb K)$, as in Proposition \ref{P_2=produto semidireto e presentacao F2 |x pi1K}.
		The following relation is induced by the relation in $\mathcal S_2$, which we compute using Relation \textit{8.} of Proposition \ref{presentation Carolina B_n(K)(teo)}, i.e. $b^{-1}\sigma^2a b a=1$;
		\begin{align}\label{equation sigma^2=...(eq)}
			\sigma^2=\sigma_1^2=\sigma_1 b^{-1}\sigma_1^2a b a\sigma_1
			=(\sigma_1 b^{-1}\sigma_1)(\sigma_1 a\sigma_1)(\sigma_1^{-1}b\sigma_1^{-1})(\sigma_1 a\sigma_1)
			=(i\circ\Theta_P)(b_2^{-1}a_2b_2a_2).
		\end{align}
		The relations \textit{7.} to \textit{10.} are computed by conjugating the generators of $F_2\rtimes_{\theta}\pi_1(\mathbb K)$, using the relations of Proposition \ref{presentation Carolina B_n(K)(teo)}, as well as the Relations \textit{1.} to \textit{6.} and Equation \eqref{equation sigma^2=...(eq)}. As an example we compute Relation \textit{10}, all others follow analogously.
		\begin{align*}
			\sigma^{-1} b_2 \sigma=& \sigma_1^{-1} (i\circ\Theta_P)(b_2)\sigma_1
			=\sigma_1^{-1} (\sigma_1^{-1} b \sigma_1^{-1} )\sigma_1	= \sigma^{-2} b
			=\sigma_1^{-2}(\sigma_1^{-1}b\sigma_1^{-1})^{-1}(\sigma_1^{-1} b \sigma_1^{-1}b)\\
			&=(i\circ\Theta_P)((a_2^{-1}b_2^{-1}a_2^{-1}b_2)b_2^{-1}b_1)
			=\sigma^{-2}(i\circ\Theta_P)(b_2^{-1}b_1).
		\end{align*}		
		We identify $a_1,b_1,a_2$ and $b_2$  with  $(i\circ\Theta_P)(a_1), (i\circ\Theta_P)(b_1),(i\circ\Theta_P)(a_2)$ and $(i\circ\Theta_P)(b_2)$ respectively and $\sigma$ with $\tau(\sigma_1)$ and obtain the wanted relations.
	\end{proof}
	
	\begin{cor}\label{relations sigma^2} 
		In the presentation of Proposition \ref{cap-braids:Presentation B_2(minha)(theo)} it holds that $(\sigma^{-1}b_1\sigma)^s=\sigma^{-1+(-1)^{s}}b_1^s$, for any $s\in \mathbb Z$.
	\end{cor}
	\begin{proof}
		Relation \textit{8.} of Proposition \ref{cap-braids:Presentation B_2(minha)(theo)} implies that $\sigma b_1=b_1\sigma^{-1}$.
		To prove the corollary show first that $(\sigma^{-2}b_1)^{2t}=b_1^{2t}$ for any $t\in \mathbb Z$, then use relation \textit{8.} of Proposition \ref{cap-braids:Presentation B_2(minha)(theo)} and divide the cases depending whether $s$ is even or odd.
	\end{proof}
	
	\begin{rem}\label{elements in B2(K)}
		\begin{enumerate}[i)]
			\item\label{thetaa^rb^s}
				Note that relations \textit{1.} to \textit{5.} are determined by the relations in $F_2(a_2,b_2)\rtimes_\theta\pi_1(\mathbb K)$, where $\theta$ is given as in Proposition \ref{P_2=produto semidireto e presentacao F2 |x pi1K}, hence for a word $w(a_2,b_2)\in F_2(a_2,b_2)$ it holds, for every $r,s\in \mathbb Z$, that $a_1^rb_1^s w(a_2,b_2)=w(\theta_{a_1^rb_1^s}(a_2),\theta_{a_1^rb_1^s}(b_2)) a_1^rb_1^s$, where $\theta_{a_1^rb_1^s}$ is as in Proposition \ref{P_2=produto semidireto e presentacao F2 |x pi1K}.
			\item \label{B2 normalform}
				From Theorem \ref{cap-braids:Presentation B_2(minha)(theo)} it follows that every element $\alpha$ in $B_2(K)$ can be written in the form $\alpha=w(a_2,b_2) a_1^rb_1^s\sigma^k$, where $w(a_2,b_2)$ is a word in $F_2(a_2,b_2)$, $a_1^rb_1^s\in \pi_1(\mathbb K)$ 
				and $k$ is either zero or one.
				Note that $k$ is zero if and only if $\alpha\in F_2(a_2,b_2)\rtimes_{\theta}\pi_1(\mathbb K)$, i.e. a pure braid.
				We call this form the \textit{normal form} in $B_2(\mathbb K)$ and it holds that two elements are equal if their normal form is equal.
				Hence to solve an equation in $B_2(\mathbb K)$  means to solve equations in $F_2(a_2,b_2)$ and in $\pi_1(\mathbb K)$
			\item\label{pr}
				To solve equations in $F_2(a_2,b_2)$ it helps to look at the equation in different quotient spaces, in this case we use the quatient space  $F_2(a_2,b_2)/\langle b_2^{-1}a_2b_2a_2 \rangle$, which is isomorphic to the fundamental group of the Klein bottle.
				Let $pr:F_2(a_2,b_2)\rightarrow F_2(a_2,b_2)/\langle b_2^{-1}a_2b_2a_2 \rangle$ be the projection to the quotient space.
				We can write the image of every word $w_2(a_2,b_2)\in F_2(a_2,b_2)$ as $pr(w(a_2,b_2))=a_2^mb_2^n$, where $m,n\in \mathbb Z$.
				Note that $n=|w(a_2,b_2)|^{b_2}$ is the sum of exponents of $b_2$ of the word $w(a_2,b_2)$, but $m$ is not necessary the sum of the exponents of $a_2$.
		\end{enumerate}
	\end{rem}

\section{Nielsen number for 2-valued non-split maps on the Klein bottle}\label{sec:Nielsennumber}
	In this section we display a formula for the Nielsen number of $2$-valued non-split maps on the Klein bottle.
	For that purpose let $\phi:\mathbb K\multimap \mathbb K$ be a $2$-valued non-split map and let $\Phi:\mathbb K\rightarrow D_2(\mathbb K)$ be the correspondent map.
	Let furthermore the induced homomorphism be as in Equation \eqref{Phi nonsplit}, i.e. as follows;
	\begin{align}
		\Phi_\#:
		\alpha &\longmapsto\hat\alpha=w_1(a_2,b_2)a_1^{r_1}b_1^{s_1}\sigma^{k_1}\label{Phi non-split}\\
		\beta &\longmapsto\hat\beta=w_2(a_2,b_2)a_1^{r_2}b_1^{s_2}\sigma^{k_2},\nonumber
	\end{align}
	where $k_1,k_2\in \mathbb Z_2$ and $r_1,r_2,s_1,s_2\in \mathbb N$, $pr(w_i)=a_2^{m_i}b_2^{n_i}$, for $i=0,1$, where $pr$ is the projection as in Remark \ref{elements in B2(K)}\ref{pr}).
	That is, we wrote $\hat \alpha$ and $\hat \beta$ as arbitrary braids in $B_2(\mathbb K)$ in their normal form, as described in Remark \ref{elements in B2(K)}\ref{B2 normalform}).
	Since we require for $\phi$ to be a map we suppose that the variables are such that $\hat \alpha\hat \beta \hat \alpha=\hat \beta$.
	Since $\phi$ isn't split, it holds that $\Phi_\#(\pi_1(\mathbb K))\not\subset P_2(\mathbb K)$, therefore at least one of the generators is mapped to a braid that isn't pure.
	Hence it can't happen that both $k_1$ and $k_2$ are equal to zero.
	We call $\phi$ of \textit{type $(A)$} if $\hat \alpha$ is pure and $\hat \beta$ is not, i.e. $k_1=0$ and $k_2=1$ and of \textit{type $(B)$} if $k_1=1$. 
	We compute the Nielsen number for map of each type in terms of $\hat \alpha$ and $\hat \beta$.
	
	\subsection{Maps of type (A)}\label{maps A}

	Let $\phi_A:\mathbb K \multimap\mathbb  K$ be a $2$-valued (non-split) map of type $(A)$ and consider the correspondent map $\Phi_A:\mathbb K\rightarrow D_2(\mathbb K)$, 
	where the induced homomorphism is as in Equation \eqref{Phi non-split}, with $k_1=0$ and $k_2=1$, i.e.
		\begin{align}
			\Phi_\#:
			\alpha &\longmapsto\hat\alpha=w_1(a_2,b_2)a_1^{r_1}b_1^{s_1}\label{Phi non-split type A}\\
			\beta &\longmapsto\hat\beta=w_2(a_2,b_2)a_1^{r_2}b_1^{s_2}\sigma,\nonumber
		\end{align}
	where $r_1,r_2,s_1,s_2\in \mathbb N$, $pr(w_i)=a_2^{m_i}b_2^{n_i}$, for $i=0,1$, where $pr:F_2(a_2,b_2)\rightarrow F_2(a_2,b_2)/ \langle a_2b_2a_2b_2^{-1}\rangle$ is the projection as in Remark \ref{elements in B2(K)}\eqref{pr}.
	We follow the construction described in Section \ref{sec:2valued maps &BU}.
	Consider the subgroup $H_A=(\Phi_A)_\#^{-1}(P_2(\mathbb K))\subseteq\pi_1(\mathbb K)$, which is kernel of the map $\rho\circ (\Phi_A)_\#$.
	Using the Reidemeister-Schreier method we find a presentation
	$H_A=\langle \alpha, \beta^2:\alpha\beta^2=\beta^2\alpha\rangle$,
	which is isomorphic to the fundamental group of the Torus $\pi_1(\mathbb T)=\langle a,b:ab=ba\rangle $, identifying $\alpha$ with $a$ and $\beta^2$ with $b$.
	The double covering $q_A$ that corresponds to $H_A$ is given by $(q_A)_\#:\pi_1(\mathbb T) \rightarrow \pi_1(\mathbb K)$ mapping $a$ to $\alpha$ and $b$ to $\beta^2$.
	Then the map $\hat \Phi_A=\{f_1,f_2\}:\tilde {\mathbb K}_B\rightarrow F_2(\mathbb K)$, where the induced homomorphism $(\hat\Phi_A)_\#$ takes $a$ to $\hat \alpha=w_1(a_2,b_2)a_1^{r_1}b_1^{s_1}$ and $b$ to $(\hat\beta)^2=(w_2(a_2,b_2)a_1^{r_2}b_1^{s_2}\sigma)^2$, is a lift of $q_A\circ\Phi_A$.
	The non-trivial deck transformation $\delta_A:\tilde {\mathbb T}\rightarrow \tilde{\mathbb T}$ of the covering $q_A:\tilde {\mathbb T}\rightarrow \mathbb K$ is a free involution and admits a lift to the plane given by $\delta_A(\tilde x,\tilde y)= (1-\tilde x,\tilde y+\frac{1}{2})$.
	Consider $(\delta_A)_\#:\pi_1(\tilde {\mathbb T},\tilde x_1)\rightarrow \pi_1(\tilde{\mathbb T}, \tilde x_2)$ and the path $\lambda:[0,1]\rightarrow \tilde{\mathbb T}$ from $\tilde x_1$ to $\tilde x_2=\delta_A(\tilde x_1)$ as in Figure \ref{decktransf Torus}.
	
		\begin{figure}[h!]\label{decktransf Torus}
		\hfill
		\begin{tikzpicture}[scale=1, very thick]
			{\draw[blue,line width=1.5pt][->>] (0,3) -- (3,3);}
				\node at (1.5,-0.3) {$a$};					
			{\draw[blue,line width=1.5pt][->>] (0,0) -- (3,0);}
			{\draw[blue, line width=1.5pt][<-] (0.05,1.5) -- (3,1.5);}
				\node at (1.5, 1.2) {$\delta_A \circ a$};	
			{\draw[red,line width=1.5pt][->] (0,0) -- (0,3);}
				\node at (3.3,1.5) {$b$};
			{\draw[red,line width=1.5pt][->] (3,0) -- (3,3);}
			{\draw[cyan,line width=1.5pt][->] (0.06,0) -- (0.06,1.45);} 
				\node at (0.3,0.7) {$\tilde \lambda$};
			{\draw[fill=black] (0,0) circle  (0.5mm);}
				\node at (-0.3,0) {$\tilde x_1$};
			{\draw[fill=black] (0,1.5) circle (0.5mm);}
				\node at (-0.7,1.5) {$\delta_A(\tilde x_1)$};
		\end{tikzpicture}
		\hspace*{\fill} 
		\caption{$\tilde{\mathbb T}$ and deck-transformation $\delta_A:\tilde{\mathbb T}\rightarrow \tilde{\mathbb T}$}
	\end{figure}
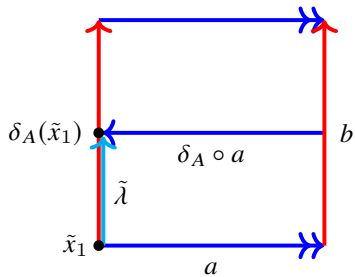
	
	Note that, abusing notation and thinking of $a$ and $b$ as the paths that represent the homotopy classes $a$ and $b$ in $\pi_1(\tilde{\mathbb T},\tilde x_1)$, it holds that $\delta_A \circ a$ and $\delta_A \circ b$ are homotopic to $\tilde \lambda^{-1}*a^{-1}*\tilde \lambda$ and $\tilde \lambda^{-1}*b*\tilde \lambda$ respectively.
	It holds that $\delta\circ a$ is homotopic to $\tilde\lambda^{-1}* a *\tilde \lambda $ and  $\delta\circ b$ is homotopic to $\tilde\lambda^{-1}* a^{-1} * b *\tilde \lambda $, where $*$ denotes concatenation and therefore, using that $f_2=f_1\circ \delta_A$ we have
	\begin{align*}
		f_2(a)&=f_1(\delta_A\circ a)\sim f_1(\tilde \lambda^{-1}* a^{-1}*  \tilde \lambda )=(f_1\circ\tilde \lambda^{-1})* (f_1\circ a^{-1})*  (f_1\circ\tilde\lambda),\\
		f_2(b)&=f_1(\delta_A\circ b)\sim f_1(\tilde \lambda^{-1}* b * \tilde \lambda )=(f_1\circ\tilde \lambda^{-1})* (f_1\circ b) * (f_1\circ\tilde\lambda). 
	\end{align*}		
	Consider the change-of-basepoint homomorphism $T:\pi_1(\mathbb K, p_1)\rightarrow \pi_1(\mathbb K, p_2)$ induced by the shortest path between $p_1=f_1(\tilde x_0)$ and $p_2=f_2(\tilde x_0)=f_1(\delta(\tilde x_0))$, then $T$ takes $a_1$ and $b_1$ to $a_2$ and $b_2$, respectively.
	\begin{figure}[h!]\label{change-of-basepoint homomorphism p1->p2, K}
		\hfill
		\begin{tikzpicture}[scale=1, very thick]
			{\draw[red,line width=1.5pt][->] (2,-3.5) -- (0,-2.5);}
			{\draw[red,line width=1.5pt][->] (0,-4.5) -- (-2,-3.5);}
			{\draw[blue,line width=1.5pt][->>] (0,-2.5) -- (-2,-3.5);} 
			{\draw[blue,line width=1.5pt][<<-] (2,-3.5) -- (0,-4.5);}
			{\draw[fill=black] (-0.5,-3.5) circle (0.5mm);}
				\node at (-0.75,-3.75) {$p_1$};	
			{\draw[fill=black] (0.5,-3.5) circle (0.5mm);}
				\node at (0.75,-3.75) {$p_2$};
			{\draw[cyan,line width=1.5pt][->] (-0.45,-3.5) --  (0.45,-3.5);}
				\node at (0,-3.2) {$\lambda$};
		\end{tikzpicture}
		\hspace*{\fill}
		\caption{path between $p_1$ and $p_2$ in $\mathbb K$}
	\end{figure}
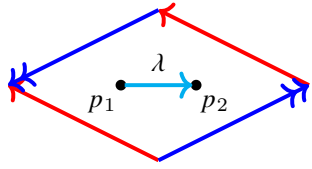
	\noindent
	Then $(f_1\circ \tilde \lambda^{-1}) *\lambda$ is a loop starting at $p_2$ and induces a homotopy class in $\pi_1(\mathbb K,p_2)$, which can be written as $a_2^kb_2^l$, for some $k,l\in \mathbb Z$.
	Then 
	\begin{align}
		\noindent
		(f_2)_\#(a)&=(a_2^kb_2^l) T((f_1)_\#(a^{-1})) (a_2^kb_2^l)^{-1},\label{eq:f2(a)=f1(a-1) Toro}	\\
		(f_2)_\#(b)&=(a_2^kb_2^l) T((f_1)_\#(b)) (a_2^kb_2^l)^{-1}.\label{eq:f2(b)=f1(b) Toro}
	\end{align}		
	
	Note that $f_1$ and $f_2$ are maps from the Torus to the Klein bottle, which are both $K(\pi,1)$-spaces, hence the maps are completely determined by their induced homomorphism.
	There exists a classification of maps from the Torus to the Klein bottle, see for example \cite[Lemma 2.1]{gonccalves2008coincidence} or \cite[Propostition 1.1.]{gonccalves2022borsuk}
	
	\begin{lem}\cite{gonccalves2008coincidence} or \cite{gonccalves2022borsuk}\label{Classification maps T->K}
		There are four types of homotopy classes (with basepoints) $[f]\in [\mathbb T, \mathbb K]$ defined by the following induced homomorphisms from $\pi_1(\mathbb T)$ to $\pi_1(\mathbb K)$, where $r_1, r_2, s_1, s_2\in \mathbb Z$.
		\begin{multicols}{4}
			
			Type 1:
			\begin{align*}
				f_\#: a&\mapsto \alpha^{r_1}\beta^{2s_1+1}\\
				b&\mapsto \beta^{2s_2},
			\end{align*}
			
			Type 2: 
			\begin{align*}
				f_\#: a&\mapsto \alpha^{r_1}\beta^{2s_1+1}\\
				b&\mapsto \alpha^{r_1}\beta^{2s_2+1},
			\end{align*}
			
			Type 3: 
			\begin{align*}
				f_\#: a&\mapsto \beta^{2s_1}\\
				b&\mapsto \alpha^{r_2}\beta^{2s_2+1},
			\end{align*}
			
			Type 4: 
			\begin{align*}
				f_\#: 
				a&\mapsto \alpha^{r_1}\beta^{2s_1}\\
				b&\mapsto \alpha^{r_2}\beta^{2s_2}.
			\end{align*}
		\end{multicols}
	\end{lem}
	
	We use the classification of homotopy classes of maps that do not satisfy the Borsuk-Ulam Property in respect to $\delta_A$, given in \cite[Theorem 1.3.]{gonccalves2022borsuk} and quote this result in the negative in the following proposition.
	
	\begin{prop}\cite[Theorem 1.3.]{gonccalves2022borsuk}\label{not BU T->K}
		Let $[f]\in [\mathbb T, \mathbb K]$, then $f$ does not satisfy the Borsuk-Ulam Property in respect to $\delta_A$ if and only if one of the following is satisfied;
			\begin{enumerate}[I)]
				\item $f$ is of type 1 and $s_2$ is odd,
				\item $f$ is of type 3 and $s_1=0$,	
				\item $f$ is of type 4 and all of these conditions hold;
				\begin{enumerate}[i)]
					\item $r_2s_1=0$, \label{T->K:r2s1=0}
					\item $r_2\neq 0$ or $s_2\neq0$ or $s_1=0$,
					\item $s_1\neq 0$ or $s_2\neq0$ or $r_1=0$ or $r_2$ is odd.
				\end{enumerate}
			\end{enumerate}
	\end{prop}

	\begin{cor}\label{not BU T->K type 4}
		The last condition of Proposition \ref{not BU T->K} can be written as
		$f$ is of type \textit{4} and one of these conditions hold;
		\begin{enumerate}[i)]
			\item $s_1=0$ and if $s_2=0$, then $r_1=0$ or $r_2$ is odd,
			\item $r_2=0$ and $s_2\neq 0$,
			\item $r_1=r_2=s_1=s_2=0$.
		\end{enumerate}
	\end{cor}
	\begin{proof}
		From \textit{IIIi)} it follows that either $r_2=0$ or $s_1=0$.
		Suppose that $s_1=0$, then conditions \textit{IIIi)} and \textit{IIIii)} hold anyway and \textit{IIIiii)} holds, if and only if, whenever $s_2=0$, then $r_1=0$ or $r_2$ is odd.
		Suppose that $r_2=0$,
		if addiotionally $s_2\neq 0$, then all conditions \textit{IIIi)}, \textit{IIIii)} and \textit{IIIiii)} are satisfied.		
		If $r_2=s_2=0$, then \textit{IIIii)} implies $s_1=0$, but since $r_2=0$ is even it follows from \textit{IIIiii)} that $r_1=0$.
	\end{proof}

	\begin{prop}\label{PhiA=>f1,f2}
		Let $\phi_A:\mathbb K\multimap \mathbb K$ be a $2$-valued map of type $(A)$, with correspondent map $\Phi_A:\mathbb K\rightarrow D_2(\mathbb K)$, where the induced homomorphism $(\Phi_A)_\#$, the covering space $q_A: \tilde{\mathbb T}\rightarrow \mathbb K$, the deck transformation $\delta_A:\tilde{\mathbb T}\rightarrow \tilde{\mathbb T}$ and the lift $\hat \Phi_A=\{f_1,f_2\}:\tilde{\mathbb T}\rightarrow F_2(\mathbb K)$
		are as in Equation \eqref{Phi non-split type A} and ff.
		Let furthermore $k,l\in \mathbb Z$ be such that
		$[(f_1\circ \tilde \lambda^{-1}) *\lambda]=a_2^kb_2^l\in \pi_1(\mathbb K,p_2)$, where $\tilde \lambda$ and $\lambda$ are depicted in Figures \ref{decktransf Torus} and \ref{change-of-basepoint homomorphism p1->p2, K}.
		Then $n_1=-2s_1$ and
		\begin{align}
			m_1&=((-1)^l-1)r_1+(1-(-1)^{s_1})k\label{eq:T m1}\\
			(1-(-1)^{l+s_2})m_2&=((-1)^{l}(1+(-1)^{s_2+n_2})-(-1)^{n_2}(1+(-1)^{s_2}))r_2+(1-(-1)^{n_2})k\label{eq:T m2}
		\end{align}
		where $pr(w_i(a_2,b_2))=a_2^{m_i}b_2^{n_i}$ with $pr$ as in Remark \ref{elements in B2(K)}\eqref{pr}.
		Furthermore $f_2=f_1\circ \delta_A$  and $(f_1)_\#:\pi_1(\mathbb T,\tilde x_1)\rightarrow \pi_1(\mathbb K,p_1)$ is given by
		\begin{align}
			(f_1)_\#:a &\longmapsto a_1^{r_1}b_1^{s_1}\label{fA}\\
				b &\longmapsto a_1^{(-1)^{s_2}m_2+(1+(-1)^{n_2+s_2})r_2}b_1^{2s_2+n_2},\nonumber
		\end{align} 
		where either $(f_1)_\#$ is the trivial homomorphism or one of the following is satisfied;
		\begin{enumerate}[I)]
			\item $f$ is of type 1 and $s_2+\frac{n_2}{2}$ is odd,
			\item $f$ is of type 3 and $s_1=0$,
			\item $f$ is of type 4 and one of these conditions hold;
				\begin{enumerate}[i)]
					\item $s_1=0$ and if $n_2=-2s_2$, then $r_1=0$ or $(-1)^{s_2}m_2+(1+(-1)^{s_2})r_2$ is odd,
					\item $m_2=-(1+(-1)^{s_2})r_2$ and $2s_2+n_2\neq 0$.
				\end{enumerate}
		\end{enumerate}
	\end{prop}
	\begin{proof}
		Let $pr_i:F_2(\mathbb K)\rightarrow \mathbb K$ be the projection to the $i$-th component, for $i=1,2$, as in Remark \ref{rem P2 normal form pri}\eqref{pr1,pr2},
		then $f_i=pr_i\circ \hat \Phi_A$.
		Note that to compute the induced homomorphism of $f_1$ and $f_2$ we have to use the isomorphism $\Theta_P:F_2(a_2,b_2)\rtimes_\theta\pi_1(\mathbb K)\rightarrow P_2(\mathbb K)$ given in Proposition \ref{P_2=produto semidireto e presentacao F2 |x pi1K}, since the domain of $(pr_i)_\#$ is $P_2(\mathbb K)$.
		Hence $(f_i)_\#=(pr_i)_\#\circ \Theta_P\circ (\hat \Phi_A)_\#$, which we compute on the generators $a$ and $b$.
		The homomorphism $(\hat\Phi_A)_\#$ takes $a$ to $\hat \alpha=w_1(a_2,b_2)a_1^{r_1}b_1^{s_1}$ and $b$ to $(\hat\beta)^2=(w_2(a_2,b_2)a_1^{r_2}b_1^{s_2}\sigma)^2$ (see the beginning of this section).
		It is considerable more straightforward to compute $(f_i)_\#(a)$, hence we only show the details to compute $(f_i)_\#(b)$.
		To be able to apply $\Theta_P$ on $(\hat\beta)^2$, we have to first write it in term of elements in $F_2(a_2,b_2)\rtimes_\theta\pi_1(\mathbb K)$.
		For that purpose note that $\sigma^2= b_2^{-1} a_2 b_2 a_2$ is a word in the free group $ F_2(a_2,b_2)$.
		For this next computation we use relations \textit{6.} to \textit{10.} in Theorem \ref{cap-braids:Presentation B_2(minha)(theo)}, as well as the relation from Corollary \ref{relations sigma^2};
		\begin{align*}
			(\hat \Phi_A)_\#(b)&=w_2(a_2,b_2)a_1^{r_2}b_1^{s_2}\sigma w_2(a_2,b_2) a_1^{r_2}b_1^{s_2}\sigma\\
			&=w_2(a_2,b_2)a_1^{r_2}b_1^{s_2}\sigma^2 w_2(\sigma^{-1}a_2\sigma,\sigma^{-1}b_2\sigma)(\sigma^{-1}a_1\sigma)^{r_2}
			(\sigma^{-1} b_1\sigma)^{s_2}\\	
			&=w_2(a_2,b_2)a_1^{r_2}b_1^{s_2}\sigma^2 w_2(a_2^{-1}\sigma^2a_1,\sigma^{-2}b_2^{-1}b_1)a_1^{r_2}
			\sigma^{-1+(-1)^{s_2}}b_1^{s_2}.
		\end{align*}
		We apply $\Theta_P$ on this element 
		and obtain the following pure braid,  
		where $\hat \sigma^2= \hat b_2^{-1}\hat a_2 \hat b_2 \hat a_2$;
		\begin{equation*}
			(\Theta_P\circ (\hat \Phi_A)_\#)(b)
			=w_2(\hat a_2,\hat b_2)(\hat a_2\hat a_1)^{r_2} (\hat b_2 \hat b_1)^{s_2}\hat\sigma^2 w_2(\hat a_2^{-1}\hat\sigma^2\hat a_2\hat a_1,\hat\sigma^{-2}\hat b_1)(\hat a_2\hat a_1)^{r_2}\hat\sigma^{-1+(-1)^{s_2}}(\hat b_2\hat b_1)^{s_2}.
		\end{equation*}
		The last step to computing $(f_1)_\#(b)$ and $(f_2)_\#(b)$ is to apply $(pr_1)_\#$ and $(pr_2)_\#$, respectively, where
		we use that  $(pr_2)_\#$ takes $\hat \sigma^2$ to $1$.
		\begin{align*}
			(f_1)_\#(b) 
			&=(a_1^{r_2}b_1^{s_2})( a_1^{m_2}b_1^{n_2})( a_1^{r_2}b_1^{s_2})
			=a_1^{ (-1)^{s_2}m_2+r_2+ (-1)^{s_2+n_2}r_2}b_1^{2s_2+n_2},\\
			(f_2)_\#(b)
			&=(a_2^{m_2}b_2^{n_2})(a_2^{r_2}b_2^{s_2})(a_2^{r_2} b_2^{s_2})
			=a_2^{m_2+(-1)^{n_2}r_2+(-1)^{s_2+n_2}r_2}b_2^{2s_2+n_2}.
		\end{align*}
		In an analogous way we come to the conclusion that $(f_1)_\#$ takes $a$ to $a_1^{r_1}b_1^{s_1}$ and $(f_2)_\#$ takes $a$ to $a_2^{m_1+(-1)^{n_1}r_1}b_2^{n_1+s_1}$.
		From Proposition \ref{Phi map=>-f_1 BU} we know that $f_2=f_1\circ \delta_A$, therefore, by Equation \eqref{eq:f2(b)=f1(b) Toro}  it follows that
		\begin{align*}
			(f_2)_\#(a)&=(a_2^kb_2^l) T((f_1)_\#(a^{-1})) (a_2^kb_2^l)^{-1}= (a_2^kb_2^l) T(a_1^{r_1}b_1^{-s_1})(a_2^kb_2^l)^{-1}\\
			&=(a_2^kb_2^l) (a_2^{r_1}b_2^{-s_1})(a_2^kb_2^{-l})=a_2^{(-1)^lr_1+2k}b_2^{-s_1},
		\end{align*}		
		Analogously, using Equation \eqref{eq:f2(a)=f1(a-1) Toro} it follows that
		$(f_2)_\#(b)
		=a_2^{(-1)^l ((-1)^{s_2}m_2+r_2+ (-1)^{s_2}r_2)+2k}b_2^{2s_2}$.
		Comparing the powers of $a_2$ and $b_2$ for both $a$ and $b$ we obtain four equations, which imply $n_1=-2s_1$ and the Equations \eqref{eq:T m1} and \eqref{eq:T m2}.
		Further, from Proposition \ref{Phi map=>-f_1 BU} it follows also that $[f_1]$ does not satisfy the Borsuk-Ulam Property in respect to $\delta_A$.
		With Proposition \ref{not BU T->K} and Corollary \ref{not BU T->K type 4} it follows straightforwardly, that $(f_1)\#$ has to be as given above in Equation \eqref{fA}.
	\end{proof}	
	
	We know from Theorem \ref{N(phi)=N(q,f1)} that a way to compute the Nielsen fixed point number of $\Phi_A$ is to compute the Nielsen coincidence number of the pair $(q_A,f_1)$.
	To compute the coincidence number we use the following formula for maps from the Torus to the Klein bottle, given in \cite{gonccalves2008coincidence}.
	We only cite the cases necessary to compute $N(q_A,f_1)$, considering that $f_1$ is of type \textit{1, 3} or \textit{4} and $q_A$ is of type \textit{4}.
	\begin{prop}\cite[Lemma 2.1]{gonccalves2008coincidence}\label{Nielsen Coincidence number T->K(prop)}
		Let $f,g:\mathbb T\rightarrow \mathbb K$ be two maps, where $f$ is as in Lemma \ref{Classification maps T->K} and $g$ of type \textit{4}, where $g_\#$ maps $a$ to $\alpha^{t_1}\beta^{2v_1}$ and $b$ to $\alpha^{t_2}\beta^{2v_2}$.
		Then the Nielsen coincidence number $N(g,f)$ is given as follows
		\begin{itemize}
			\item if $f$ of type \textit{1}, then
					$N(g,f)=N(f,g)=|t_1(2v_2-2s_2)-t_2(2v_1-2s_1-1)|,$
			\item if $f$ of type \textit{3}, then
					$N(g,f)=N(f,g)=|t_1(2v_2-2s_2-1)-t_2(2v_1-2s_1)|,$
			\item\label{N(q,f)T->K: type 4} if $f$ is of type \textit{4}, then $N(g,f)=N(f,g)$ is equal to
			\begin{align*}
				|(r_1-t_1)(s_2-v_2)-(r_2-t_2)(s_1-v_1)|+|(r_1+t_1)(s_2-v_2)-(r_2+t_2)(s_1-v_1)|.
			\end{align*}
		\end{itemize}
	\end{prop}

	\begin{prop}\label{Nielsennumber A}
		Let $\phi:K\multimap K$ of type $(A)$, i.e. $(\Phi_A)_\#$ maps $\alpha$ to $w_1(a_2,b_2)a_1^{r_1}b_1^{s_1}$ and $\beta$ to $w_2(a_2,b_2)a_1^{r_2}b_1^{s_2}\sigma$.
		Let $n_2=|w_2|^{b_2}$ be the sum of exponents of $b_2$ in $w_2(a_2,b_2)$. Then the Nielsen fixed point number of $\phi$ is given as
		\begin{align*}
			N(\phi)=
			\begin{cases}
				|r_1|\cdot|2s_2+n_2-2|& \text{ if } s_1,n_2 \text{ even and } r_1\neq 0 \\
				|2s_2+n_2-2| & \text{ otherwise.}
			\end{cases}
		\end{align*}
	\end{prop}
	
	\begin{proof}
		We know from Theorem \ref{N(phi)=N(q,f1)} that the Nielsen fixed point number $N(\phi)$ of $\phi$ is equal to the Nielsen coincidence number $N(q_A,f_1)$ of the pair $(q_A,f_1)$.
		Note that $(q_A)_\#$ maps $a$ to $\alpha$ and $b$ to $\beta^2$, i.e. $q_A$ is of type \textit{4} and it holds that $t_1=v_2=1$ and $t_2=v_1=0$.
		The induced homomorphism $(f_1)_\#$ is given as in Proposition \ref{PhiA=>f1,f2}, hence is of type \textit{1}, \textit{3} or \textit{4}.
		Let $f_1$ be of type \textit{4}, hence $s_1$ and $n_2$ are both even.
		Then setting all variables (taking care with the different notations) in the third formula of Proposition \ref{Nielsen Coincidence number T->K(prop)} we obtain, using \textit{\ref{T->K:r2s1=0})} in Proposition \ref{not BU T->K}, the following formula;
		\begin{align*}
			N(q_A,f_1)&=|(r_1-1)(s_2+\frac{n_2}{2}-1)-0|+|(r_1+1)(s_2+\frac{n_2}{2}-1)-0|\\
			&=|s_2+\frac{n_2}{2}-1|\cdot(|r_1-1|+|r_1+1|).
		\end{align*}
		Observe that if $r_1=0$, then $|r_1-1|+|r_1+1|=2$ and consequently $N(q_A,f_1)=|2s_2+n_2-2|$.
		If $r_1\neq 0$ it is straightforward to show that $|r_1-1|+|r_1+1|=2|r_1|$ and therefore $N(q_A,f_1)=|r_1|\cdot|2s_2+n_2-2|$.
		For the other cases, i.e. for $f_1$ of type \textit{1} or \textit{3}, 
		use the formulas given in Proposition \ref{Nielsen Coincidence number T->K(prop)} analogously.
		Recall, if $f_1$ is of type \textit{3}, to think of $2s_2+n_2$ as $2(s_2+\frac{n_2-1}{2})+1$ in order for it to be in the form given in Lemma \ref{Classification maps T->K}.
		In both cases we obtain that $N(q_A,f_1)=|2-2s_2-n_2|=|2s_2+n_2-2|$.
	\end{proof}


    \subsection{Maps of type (B)}
		Analogously as in the last section for maps of type $(A)$, we compute the lift factors and Nielsen number for maps of type $(B)$.
		Let $\phi_B:\mathbb K \multimap \mathbb K$ be a $2$-valued map of type $(B)$, where the induced homomorphism of the correspondent map $\Phi_B$ is given as in Equation \eqref{Phi non-split} with $k_1=1$, i.e.
		\begin{align}
			(\Phi_{B_k})_\#:
			\alpha &\longmapsto\hat\alpha=w_1(a_2,b_2)a_1^{r_1}b_1^{s_1}\sigma\label{Phi non-split type B}\\
			\beta &\longmapsto\hat\beta=w_2(a_2,b_2)a_1^{r_2}b_1^{s_2}\sigma^{k},\nonumber
		\end{align}
		where $r_1,r_2,s_1,s_2\in \mathbb N$, $pr(w_i)=a_2^{m_i}b_2^{n_i}$, for $i=0,1$, with $pr$ as in Remark \ref{elements in B2(K)}\textit{\ref{pr})}.
		Furthermore $k$ is either $0$ or $1$ and we divide it further in two cases, depending on whether $k$ is $0$ or $1$ we say $\phi_B$ is of type $(B_0)$ or $(B_1)$, respectively.
		
		We start with maps type $(B_0)$ and subsequently show how maps of type $(B_1)$ can be derived from maps of type $(B_0)$ through a homeomorphism.
		Again we follow the construction described in Section \ref{sec:2valued maps &BU} and start by computing a presentation for $H_{B_0}=ker \theta = \Phi_{B_0}^{-1}(P_2(\mathbb K))$ with the Reidemeister-Schreier method.
		We obtain $H_{B_0}=\langle \alpha^2, \beta:\alpha^2(\alpha\beta)\alpha^2=\alpha\beta\rangle$, which is isomorphic to the fundamental group of a Klein bottle $\tilde {\mathbb K}$, where $\pi_1(\tilde {\mathbb K})=\langle a,b:aba=b\rangle,$ identifying $\alpha^2$ with $a$ and $\beta$ with $b$.
 		The covering map $q_{B_0}:\tilde {\mathbb K}_B\rightarrow \mathbb K$ takes $a$ to $\alpha^2$ and $b$ to $\beta$.
 		The unique non-trivial deck transformation $\delta_{B_0}:\tilde {\mathbb K}\rightarrow \tilde {\mathbb K}$ lifts to a homeomorphism $\hat \delta_{B_0}:\mathbb R^2\rightarrow \mathbb R^2$ such that $\hat \delta_{B_0}(x,y)=(x+\frac{1}{2},y)$.
 		Let $\hat \Phi_{B_0}=\{f_1, f_2\}:\tilde {\mathbb K}\rightarrow F_2(\mathbb K)$ be such that $(\hat \Phi_{B_0})_\#$
 		takes $a$ to $(w_1(a_2,b_2)a_1^{r_1}b_1^{s_1}\sigma)^2$ and $b$ to $w_2(a_2,b_2)a_1^{r_2}b_1^{s_2}$, then $\hat \Phi_{B_0}$ is a lift of $q_{B_0}\circ\Phi_{B_0}$.

		If $\phi_{B_1}$ is of type $(B_1)$, then $H_{B_1}=\langle \alpha^2, \beta:\alpha^2\beta\alpha^2=\beta\rangle$,
		which is isomorphic to the fundamental group of the Klein bottle $\pi_1(\tilde{\mathbb K})=\langle \tilde a,\tilde b:\tilde a\tilde b\tilde a=\tilde b\rangle $, identifying $\alpha^2$ with $\tilde a$ and $\alpha\beta$ with $\tilde b$.
		The covering map $q_{B_1}:\tilde {\mathbb K}_B\rightarrow \mathbb K$ takes $a$ to $\alpha^2$ and $b$ to $\alpha\beta$.
		Let $\hat \Phi_{B_1}=\{f_1, f_2\}:\tilde {\mathbb K}\rightarrow F_2(\mathbb K)$ be such that $(\hat \Phi_{B_1})_\#$
		takes $a$ to $(w_1(a_2,b_2)a_1^{r_1}b_1^{s_1}\sigma)^2$ and $b$ to $(w_1(a_2,b_2)a_1^{r_1}b_1^{s_1}\sigma)(w_2(a_2,b_2)a_1^{r_2}b_1^{s_2}\sigma)$, then $\hat \Phi_{B_1}$ is a lift of $q_{B_1}\circ\Phi_{B_1}$.
		The unique non-trivial deck transformation $\delta_{B_1}:\tilde {\mathbb K}\rightarrow \tilde {\mathbb K}$ lifts to a homeomorphism $\hat \delta_{B_1}:\mathbb R^2\rightarrow \mathbb R^2$ taking $\hat \delta_{B_1}(x,y)=(x+\frac{1}{2},y)$.
		Note that $\delta_{B_0}$ and $\delta_{B_1}$ are the same map and we write $\delta_B$ from here on.
		
		\begin{equation*}
			\begin{tikzcd}[column sep=huge]
				\tilde {\mathbb K} \arrow{d}[swap]{q_{B_0}}\arrow{r}{id}&\tilde {\mathbb K}\arrow{r}{\hat \Phi_{B_1}}\arrow{d}{q_{B_1}}&F_2(\tilde {\mathbb K})\arrow{d}\\
				\mathbb K \arrow{r}[swap]{\Psi}&\mathbb K\arrow{r}[swap]{\Phi_{B_1}} & D_2(\mathbb K).
			\end{tikzcd}
		\end{equation*}		
		We can connect every $\Phi_{B_1}$ of type $(B_1)$ to a map of type $(B_0)$, considering the homeomorphism $\Psi:\mathbb K\rightarrow \mathbb K$, which induced isomorphism $\Psi_\#$ takes $\alpha$ to $\alpha$ and $\beta$ to $\alpha\beta$.
		We can bijectively map every $\Phi_{B_1}$ of type $(B_1)$ to a map $\Phi_{B_1}\circ\Psi$ of type $(B_0)$.
		Further, there is a lift of $\Psi\circ q_{B_0}$, which is the identity on $\tilde {\mathbb K}$.
		Since the identity is trivially $\delta_B$-equivariant (i.e. $id\circ\delta_B=\delta_B\circ id$), every result we get in relation to the Borsuk-Ulam Property for lift factors of maps of type $(B_0)$, we can also use for the lift factors of type $(B_1)$.
		Furthermore, since the lift of $\Psi\circ q_{B_0}$ is the identity, it holds that $\hat \Phi_{B_1}=\{f_1, f_2\}$, a lift of $\Phi_{B_1}\circ q_{B_1}$, is also a lift of $(\Phi_{B_1}\circ\Psi)\circ q_{B_0}$.
		In our next step we won't use that isomorphism, because in this specific case the computations are a little bit more straightforward if we continue as for maps of type $(A)$. 
		
		Consider the homomorphism $(\delta_B)_\#:\pi_1(\mathbb K, \tilde x_1)\rightarrow \pi_1(\mathbb K, \delta(\tilde x_1))$ induced by the deck transformation $\delta_B$ and consider the following figure, where we write $a$ and $b$ for the loops that represent the homotopy classes $a$ and $b$ abusing notation and $\tilde \lambda$ is the path between $(\tilde x_0)$ and $\delta_B(\tilde x_0)$ as shown in the figure below.
		\begin{figure}[h!]\label{decktransf Klein}
			\hfill
			\begin{tikzpicture}[scale=1, very thick]
				{\draw[cyan,line width=1.5pt][<-] (1.5,0.06) -- (3,0.06);} 
					\node at (2.1,0.3) {$\tilde \lambda$};
				{\draw[blue,line width=1.5pt][<<-] (0,3) -- (3,3);}
					\node at (1.5,3.3) {$a$};					
				{\draw[blue,line width=1.5pt][->>] (0,0) -- (3,0);}
				{\draw[red,line width=1.5pt][->] (0,0) -- (0,3);}
					\node at (3.3,1.5) {$b$};
				{\draw[red,line width=1.5pt][->] (3,0) -- (3,3);}
				{\draw[red, line width=1.5pt][->] (1.5,-0.05) -- (1.5,3);}
					\node at (1,1.5) {$\delta_B\circ b$};
				{\draw[fill=black] (3,0) circle  (0.5mm);}
					\node at (3.1,-0.4) {$\tilde x_1$};
				{\draw[fill=black] (1.5,0) circle (0.5mm);}
					\node at (1.5,-0.4) {$\delta_B(\tilde x_1)$};
			\end{tikzpicture}
			\hspace*{\fill} 
			\caption{$\tilde{\mathbb K}$ and deck-transformation $\delta_B$}
		\end{figure}
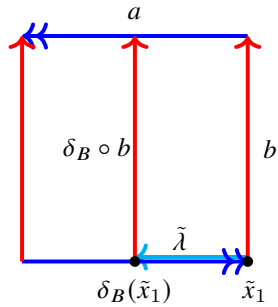
		It holds that $\delta_B\circ a$ is homotopic to $\tilde\lambda^{-1}* a *\tilde \lambda $ and  $\delta_B\circ b$ is homotopic to $\tilde\lambda^{-1}* a^{-1} * b *\tilde \lambda $, where $*$ denotes concatenation and therefore it follows from  $f_2=f_1\circ \delta$ that
		\begin{align*}
			f_2(a)&=f_1(\delta_B\circ a)\sim f_1(\tilde \lambda^{-1}* a * \tilde \lambda )=(f_1\circ\tilde \lambda^{-1})* (f_1\circ a) * (f_1\circ\tilde\lambda),\\
			f_2(b)&=f_1(\delta_B\circ b)\sim f_1(\tilde \lambda^{-1}* a^{-1}* b * \tilde \lambda )=(f_1\circ\tilde \lambda^{-1})* (f_1\circ a^{-1} f_1\circ b )* (f_1\circ\tilde\lambda). 
		\end{align*}
		Consider the change-of-basepoint homomorphism $T:\pi_1(\mathbb K, p_1)\rightarrow \pi_1(\mathbb K, p_2)$ induced by the shortest path between $p_1=f_1(\tilde x_0)$ and $p_2=f_2(\tilde x_1)=f_1(\delta(\tilde x_1))$, as in Figure \ref{change-of-basepoint homomorphism p1->p2, K}, then $T$ takes $a_1$ and $b_1$ to $a_2$ and $b_2$, respectively.
		Analogously to the beginning of Section \ref{maps A}, it holds that $(f_2)_\#$ takes the generators of $\pi_1(\mathbb K,p_2)$ to
		\begin{align}
			\noindent
			(f_2)_\#(a)&=(a_2^kb_2^l) T((f_1)_\#(a)) (a_2^kb_2^l)^{-1},\label{eq:f2(a)=f1(a-1) K}\\
			(f_2)_\#(b)&=(a_2^kb_2^l) T((f_1)_\#(a^{-1}b)) (a_2^kb_2^l)^{-1}\label{eq:f2(b)=f1(b) K},
		\end{align}				
		where $k,l\in \mathbb Z$ such that $[(f_1\circ \tilde \lambda^{-1}) *\lambda]=a_2^kb_2^l$.

		To analyze the lift $\hat \Phi_{B_0}=\{f_1, f_2\}$ and lift factors $f_1$ and $f_2$, recall that Proposition \ref{Phi map=>-f_1 BU} implies that the class of the map $f_1:\tilde {\mathbb K}\rightarrow \mathbb K$ does not satisfy the Borsuk-Ulam Property in respect to $\delta_B$.
		Hence we use the classification, given in \cite{gonccalves2019borsuk}, of homotopy classes of self-maps of the Klein bottle that do not have the Borsuk-Ulam Property in respect to the free involution $\delta_B$.
		
		\begin{lem}\cite[Propositions 38, 39, 35 and 7]{gonccalves2019borsuk}\label{K->K: not BU} 
			A based homotopy class $[f]\in [\mathbb K,\mathbb K]_0$ does not have the Borsuk-Ulam Property in respect to $\delta_B$ if and only if there are integers  $r,s,t\in \mathbb Z$ such that
			$f_\#:\pi_1(\mathbb K)\rightarrow\pi_1(\mathbb K)$ takes $\alpha$ to $\alpha^{r}$ and 
			$\beta$ to $ \alpha^{s} \beta^{2t+1}$.
		\end{lem}

		In the next two proposition we compute the (based) homotopy classes of $f_1$ (respectively of $f_2$) for maps of type $(B_0)$ and $(B_1)$ respectively.	
		\begin{prop}\label{PhiB0=>f1,f2}
			Let $\phi_{B_0}:\mathbb K\multimap \mathbb K$ be a $2$-valued map of type $(B_0)$, with the correspondent map $\Phi_{B_0}:\mathbb K\rightarrow D_2(\mathbb K)$, which takes $\alpha$ to $\hat \alpha=w_1(a_2,b_2)a_1^{r_1}b_1^{s_1}\sigma$ and $\beta$ to $\hat\beta=w_2(a_2,b_2)a_1^{r_2}b_1^{s_2}$,
			where $r_1,r_2,s_1,s_2\in \mathbb N$, $pr(w_i)=a_2^{m_i}b_2^{n_i}$, for $i=0,1$, with $pr$ as in Remark \ref{elements in B2(K)}\textit{\ref{pr})}.
			Let $q_{B_0}:\tilde{\mathbb K}\rightarrow \mathbb K$, $\hat\Phi_{B_0}=\{f_1, f_2\}$ and $\delta_B:\tilde{\mathbb K}\rightarrow \tilde{\mathbb K}$ be as in the beginning of this section, i.e. as in Equation \eqref{Phi non-split type B} and ff.
			Let $[(f_1\circ \tilde \lambda^{-1})*\lambda]=a_2^kb_2^l\in \pi_1(\mathbb K, p_2)$, where $\tilde \lambda$ is as in Figure \ref{decktransf Klein} and $\lambda$ as in Figure \ref{change-of-basepoint homomorphism p1->p2, K}.
			Then $s_2$ is odd, $n_1=-2s_1$ and it holds
			\begin{align}
				(1-(-1)^{s_1+l})m_1&=((-1)^l-1)(1+(-1)^{s_1})r_1,\label{eqB0 m1}\\
				m_2&=-(-1)^{s_1+l}m_1-(-1)^l(1+(-1)^{s_1})r_1+2k. \label{eqB0 m2}
			\end{align}
			Furthermore	$f_2=f_1\circ \delta_B$ and	$(f_1)_\#:\pi_1(\tilde {\mathbb K})\rightarrow \pi_1(\mathbb K)$ is given by
			\begin{align}\label{sec-Nn-betapure:Phi<=>f1: f1(eq in prop)}
				(f_1)_\#:a &\longmapsto a_1^{(-1)^{s_1}m_1+(1+(-1)^{s_1})r_1},\\
				b &\longmapsto a_1^{r_2}b_1^{s_2}\nonumber.
			\end{align}
		\end{prop}
		\begin{proof}
			Analogous to the proof of Proposition \ref{PhiA=>f1,f2} it holds that $f_i=pr_i\circ \hat \Phi_{B_0}:\tilde {\mathbb K}\rightarrow \mathbb K$,
			where and $pr_i:F_2(\mathbb K)\rightarrow \mathbb K$, the projection to the $i$-th component as in Remark \ref{rem P2 normal form pri}\textit{\ref{pr1,pr2})}.
			Then $(f_i)_\#=(pr_i)_\#\circ \Theta_P \circ (\hat \Phi_{B_0})_\#,$ where $\Theta_P:F_2(a_2,b_2)\rtimes\pi_1(\mathbb K)\rightarrow P_2(\mathbb K)$ is the isomorphism given in Proposition \ref{P_2=produto semidireto e presentacao F2 |x pi1K}.
			Making the computation analogously to the proof of Proposition \ref{PhiA=>f1,f2} we obtain
			\begin{multicols}{2}
				\noindent
				\begin{align*}
					(f_1)_\#:
					a &\longmapsto a_1^{(-1)^{s_1}m_1+(1+(-1)^{n_1+s_1})r_1}b_1^{2s_1+n_1}\\
					b &\longmapsto a_1^{r_2}b_1^{s_2}
				\end{align*}
				
				\noindent
				\begin{align}
					(f_2)_\#:
					a &\longmapsto a_2^{m_1+((-1)^{n_1}+ (-1)^{s_1+n_1})r_1}b_2^{2s_1+n_1}\nonumber\\
					b &\longmapsto a_2^{m_2+(-1)^{n_2}r_2}b_2^{s_2+n_2}.\label{K->K:B0 f2}
				\end{align}
			\end{multicols}
			
			Since, by Theorem \ref{Phi map=>-f_1 BU},it holds that $[f_1]$ does not have the Borsuk-Ulam Property in respect to $\delta_B$ it follows from Lemma \ref{K->K: not BU} that $n_1=-2s_1$ and $s_2$ is odd.
			Theorem \ref{Phi map=>-f_1 BU} states also that $f_2=f_1\circ\delta_B$, which, by Equations \eqref{eq:f2(a)=f1(a-1) K} and \eqref{eq:f2(b)=f1(b) K} implies that
			\begin{align*}
				(f_2)_\#(a)
					&=a_2^{({-1})^{s_1+l}m_1+(-1)^{l}(1+(-1)^{s_1})r_1} \\
				(f_2)_\#(b)
					&=a_2^{-(-1)^{s_1+l}m_1-(-1)^l(1+(-1)^{s_1})r_1+r_2+2k}b_2^{s_2}.
			\end{align*}		
			Comparing these identities with $(f_2)_\#$ given in Equation \eqref{K->K:B0 f2} implies that $n_2=0$ and 
			Equations \eqref{eqB0 m1} and \eqref{eqB0 m2} hold.
		\end{proof}		
		
		\begin{prop}\label{PhiB1=>f1,f2}
			Let $\phi_{B_1}:\mathbb K\multimap \mathbb K$ be a $2$-valued (non-split) map of type $(B_1)$, with the correspondent map $\Phi_{B_1}:\mathbb K\rightarrow D_2(\mathbb K)$, which induced homomorphism takes $\alpha$ to $w_1(a_2,b_2)a_1^{r_1}b_1^{s_1}\sigma$ and $\beta$ to $w_2(a_2,b_2)a_1^{r_2}b_1^{s_2}\sigma$,
			where $r_1,r_2,s_1,s_2\in \mathbb N$ and $pr(w_i)=a_2^{m_i}b_2^{n_i}$, where $pr$ is as in Remark \ref{elements in B2(K)}\textit{\ref{pr})}.
			Let $q_{B_1}:\tilde{\mathbb K}\rightarrow \mathbb K$, $\hat\Phi_{B_1}=\{f_1, f_2\}$ and $\delta_B:\tilde{\mathbb K}\rightarrow \tilde{\mathbb K}$ be as in the beginning of this section.
			Additionally, let $k,l\in \mathbb Z$ such that
			$[(f_1\circ \tilde \lambda^{-1}) *\lambda]=a_2^kb_2^l\in \pi_1(\mathbb K,p_2)$, where $\tilde \lambda$ and $\lambda$ are depicted in Figures \ref{decktransf Klein} and \ref{change-of-basepoint homomorphism p1->p2, K}, respectively.
			Then $s_2-s_1$ is odd, $n_1=n_2=-2s_1$ and the following Equation hold;
			\begin{align}
				(1-(-1)^{s_1+l})m_1&=((-1)^l-1)(1+(-1)^{s_1})r_1,\label{eqB1 m1}\\
				m_2&=((-1)^{s_1+l}+1)m_1+((-1)^{s_1+l}+1)r_1+((-1)^{l}-1)r_2-(-1)^l2k. \label{eqB1 m2}
			\end{align}
			Furthermore	$f_2=f_1\circ \delta_B$ and	$(f_1)_\#:\pi_1(\tilde {\mathbb K})\rightarrow \pi_1(\mathbb K)$ is given by
			\begin{align*}
				(f_1)_\#:
				a &\longmapsto a_1^{(1+(-1)^{s_1}) r_1+(-1)^{s_1} m_1}\\
				b &\longmapsto a_1^{r_1+(-1)^{ s_1} m_2+(-1)^{ s_1} r_2}b_1^{s_2-s_1}.
			\end{align*}
		\end{prop}
		\begin{proof}
			Analogously to the proof of Proposition \ref{PhiA=>f1,f2}, we come to the conclusion that
			\begin{multicols}{2}
				\noindent
				\begin{align*}
					(f_1)_\#:
					a &\longmapsto a_1^{(-1)^{s_1}m_1+(1+(-1)^{s_1+n_1})r_1}b_2^{2s_1+n_1}\\
					b &\longmapsto a_1^{(-1)^{s_1}m_2+r_1+(-1)^{s_1+n_2}r_2} b_1^{s_1+s_2+n_2}
				\end{align*}
				
				\noindent
				\begin{align}
					(f_2)_\#:
					a &\longmapsto a_2^{m_1+((-1)^{n_1}+ (-1)^{s_1+n_1})r_1}b_2^{2s_1+n_1}\nonumber\\
					b &\longmapsto a_2^{m_1+(-1)^{n_1}r_1+(-1)^{s_1+n_1}r_2}b_2^{s_1+s_2+n_1}.\label{K->K:B1 f2}
				\end{align}
			\end{multicols}
			\noindent
			Using Theorem \ref{Phi map=>-f_1 BU} we know that $[f_1]$ does not have the Borsuk-Ulam Property in respect to $\delta_B$ and it follows that $n_1=-2s_1$ and $s_1+s_2+n_2$ is odd.
			Theorem \ref{Phi map=>-f_1 BU} also states that $f_2=f_1\circ\delta_B$ and using Equations \eqref{eq:f2(a)=f1(a-1) K} and \eqref{eq:f2(b)=f1(b) K} it follows that 
			\begin{align*}
				(f_2)_\#(a)&=a_2^{({-1})^{s_1+l}m_1+(-1)^{l}(1+(-1)^{s_1})r_1}, \\
				(f_2)_\#(b)	&=a_2^{(-1)^{s_1+l}(m_2-m_1-(-1)^{n_1})r_2-r_1)+2k} b_2^{s_1+s_2+n_2}.
			\end{align*}
			But $(f_2)_\#(a)$ and $(f_2)_\#(b)$ are already given in Equation \eqref{K->K:B1 f2}.
			Comparing these two identities gives us
			that $n_1=n_2$, i.e. equal to $-2s_1$ and therefore  $s_2-s_1$ is odd.
			Furthermore Equation \eqref{eqB1 m1} and \eqref{eqB1 m2} hold.
		\end{proof}

		The Nielsen Coincidence number of self-maps on the Klein bottle has been computed and is given by the formula in the following proposition.
		\begin{prop}\cite{dobrenko1993coincidence}\label{cap-multimap:q2K->K:N(phi)(prop)}
			Let two maps $f_1,f_2:\mathbb K\rightarrow \mathbb K$ such that the induced homomorphism $(f_i)_\#$ takes $\alpha$ to $\alpha^{r_i}$
			and $\beta$ to $\alpha^{s_i}\beta^{t_i}.$
			Then the Nielsen Coincidence number is given by
			\begin{equation*}
				N(f_1,f_2)=|t_1-t_2|\text{ max }\{|r_1|,|r_2|\}.
			\end{equation*}
		\end{prop}		
		In the next proposition we compute the Nielsen number of a map of type $(B)$.
		\begin{prop}\label{Nielsennumber B}
			Let $\phi_B:\mathbb K\multimap \mathbb K$ be of type $(B)$, where $(\Phi_B)_\#$ takes $\alpha$ to $w_1(a_2,b_2)a_1^{r_1}b_1^{s_1}\sigma$ and $\beta$ to $w_2(a_2,b_2)a_1^{r_2}b_1^{s_2}\sigma^{k_2}$,
			with $r_1,r_2,s_1,s_2\in \mathbb N$, $k_2\in \{0,1\}$ and $pr(w_i)=a_2^{m_i}b_2^{n_i}$, where $pr$ is as in Remark \ref{elements in B2(K)}\eqref{pr}.
			Then
			\begin{align*}
				N(\phi)=
				\begin{cases}
					|1-s_2|\cdot max\{|(1+(-1)^{ s_1}) r_1+m_1|,2\}& \text{ if } k_2=0, \\  
					|1-s_2+s_1|\cdot max\{|(1+(-1)^{ s_1}) r_1+m_1|,2\}& \text{ if } k_2=1.
				\end{cases}
			\end{align*}
		\end{prop}
		\begin{proof}
			From Proposition \ref{N(phi)=N(q,f1)} it follows that $N(\phi_B)=N(f_1,q_B)=N(f_2,q_B)$, where $f_1,f_2$ are the lift factors and $q_B:\tilde{\mathbb K}\rightarrow \mathbb K$ the covering map.
			We will prove it in two cases, depending whether $k_2=0$ or $k_2=1$.
h			
			If $k_2=0$, then our map is of type $(B_0)$ and the induced homomorphism of the covering map $q_{B_0}:\tilde{\mathbb K}\rightarrow \mathbb K$ is given by the homomorphism takes $a $ to $\alpha^2$ and $b$ to $\beta$.
			Propoosition \ref{PhiB0=>f1,f2} implies that
			$ f_1:\tilde{\mathbb K}\rightarrow \mathbb K$ is such that $(f_1)_\#$ takes $a$ to $a_1^{(1+(-1)^{s_1})r_1+(-1)^{s_1}m_1}$ and
			$b$ to $a_1^{r_2}b_1^{s_2}.$
			By Proposition \ref{cap-multimap:q2K->K:N(phi)(prop)}
			the Nielsen coincidence number is given by
			\begin{align*}
				N(f_1,q_{B_0})&=|1-s_2|\text{ max }\{|(1+(-1)^{s_1})r_1+(-1)^{s_1}m_1|,|2|\}\\
				&=|1-s_2|\text{ max }\{|(-1)^{s_1}|\cdot|((-1)^{s_1}+1)r_1+m_1|,2\}\\
				&=|1-s_2|\text{ max }\{|(1+(-1)^{s_1})r_1+m_1|,2\}.
			\end{align*}
			
			Analogously, if $k_2=1$ the map is of type $(B_1)$ and $q_{B_1}:\tilde{\mathbb K}\rightarrow \mathbb K$ satisfies $(q_{B_1})_\#(a)=\alpha^2$ and $(q_{B_1})_\#(b)=\alpha \beta$ and for the lift factor $f_1$ it holds by Propoosition \ref{PhiB1=>f1,f2} that $(f_1)_\#$ takes $a$ to $a_1^{(1+(-1)^{s_1}) r_1+(-1)^{s_1} m_1}$ and $b$ to $a_1^{r_1+(-1)^{ s_1} m_2+(-1)^{ s_1} r_2}b_1^{s_2-s_1}$.
			Therefore 
			\begin{align*}
				N(f_1,q_{B_1})&=|1-s_2+s_1|\text{ max }\{|(1+(-1)^{s_1})r_1+m_1|,2\}.
			\end{align*}
		\end{proof}
		
		With this last part we conclude the computation of the Nielsen number of $2$-valued non-split maps on the Klein bottle.
		\begin{theo}\label{Nielsennumber non-split}
			Let $\phi:\mathbb K\multimap \mathbb K$ be a $2$-valued non-split map on the Klein bottle $\mathbb K$, where $\pi_1(\mathbb K)=\langle \alpha,\beta:\alpha\beta\alpha=\beta\rangle$ and let $\Phi:\mathbb K\rightarrow D_2(\mathbb K)$ be the correspondent map.
			Let the induced homomorphism be written as
			\begin{align*}
				\Phi_\#: 
				\alpha &\longmapsto w_1(a_2,b_2)a_1^{r_1}b_1^{s_1}\sigma^{k_1}\\
				\beta &\longmapsto w_2(a_2,b_2)a_1^{r_2}b_1^{s_2}\sigma^{k_2},\nonumber
			\end{align*}
			where 
			$k_1,k_2\in \{0,1\}$, $r_1,r_2,s_1,s_2\in \mathbb N$ and $pr(w_i)=a_2^{m_i}b_2^{n_i}$, for $i=0,1$, where $pr:F_2(a_2,b_2)\rightarrow  F_2(a_2,b_2)/\langle a_2b_2a_2b_2^{-1}\rangle$ is the projection as in Remark \ref{elements in B2(K)}\itshape{\ref{pr})}.
			Then the Nielsen fixed point number of $\phi$ is given as
			\begin{align*}
				N(\phi)=
				\begin{cases}
					|1-s_2|\cdot max\{|(1+(-1)^{ s_1}) r_1+m_1|,2\}& \text{ if } k_1=1, k_2=0, \\  
					|1-s_2+s_1|\cdot max\{|(1+(-1)^{ s_1}) r_1+m_1|,2\}& \text{ if } k_1=1, k_2=1, \\
					|r_1|\cdot|2s_2+n_2-2|& \text{ if } k_1=0, s_1,n_2 \text{ even and } r_1\neq 0, \\
					|2s_2+n_2-2| & \text{ otherwise.}
				\end{cases}
			\end{align*}
		\end{theo}		
		\begin{proof}
			Every $2$-valued non-split is either of type $(A)$ or $(B)$.
			Combining Propositions \ref{Nielsennumber A} and \ref{Nielsennumber B} gives us the result.
		\end{proof}
		

	\section{Final comments}		
		
		Note that in Theorem \ref{Nielsennumber non-split} we already assumed that $\Phi:\mathbb K\rightarrow D_2(\mathbb K)$ is a map, where we write $\Phi(\alpha)$ and $\Phi(\beta)$ in the normal form using the variables $w_1(a_2,b_2),w_2(a_2,b_2)\in F_2(a_2,b_2)$, $r_1,s_1,r_2,s_2\in \mathbb Z$ and $k_1,k_2\in \{0,1\}$.
		We find some restriction for these variables when we analyze the lift factors in Propositions \ref{PhiA=>f1,f2}, \ref{PhiB0=>f1,f2} and \ref{PhiB1=>f1,f2}.
		A project that is still open is to classify these maps, i.e. $2$-valued non-split maps on the Klein-bottle, for example in terms of the variables above.
		But since this is a big question, we can make some smaller steps.		
		We know that every $2$-valued non-split self-map on the Klein bottle induces a pair of lift-factors (see paragraphs just before Theorem \ref{fsigmak=fdeltak}), which are either maps from the Torus to the Klein bottle (in case of maps of type $(A)$) or self-maps on the Klein bottle (in case of maps of type $(B)$).
		With this correspondence in mind, we can define maps $\Omega_A$ (respectively $\Omega_B$) that go from the based homotopy classes $[\mathbb K, D_2(\mathbb K)]_0$ to $[\tilde {\mathbb T}, \mathbb K]_0\times[\tilde {\mathbb T}, \mathbb K]_0$ (respectively to $[\tilde {\mathbb K}, \mathbb K]_0\times[\tilde {\mathbb K}, \mathbb K]_0$).
		From Remark \ref{fBU->2valued non split map} it follows that the image of $\Omega_A$ are all elements in the form $([f],[f\circ \delta_A])$, where $[f]$ is a class that satisfies the conditions of Proposition \ref{not BU T->K}, i.e. does not satisfy the Borsuk-Ulam Property in respect to $\delta_A$.
		Analogously the image of $\Omega_B$ are all $([f],[f\circ \delta_B])$ such that $[f]$ satisfies the conditions of Proposition \ref{K->K: not BU}.
		One first step to classify $2$-valued non-split maps on the Klein-bottle could be to fix a pair $([f],[f\circ\delta_A])$ (respectively $([f],[f\circ\delta_B])$), where $[f]$ does not satisfy the Borsuk-Ulam Property in respect to $\delta_A$ (respectively $\delta_B$), which are given in
		Proposition \ref{not BU T->K} (respectively Proposition \ref{K->K: not BU}),
		and classify all $2$-valued non-split maps that are in the pre-image of this pair.
		There are examples of $2$-valued non-split maps on the Klein-bottle on the pre-image of each such a pair $([f],[f\circ\delta_A])$ (respectively $([f],[f\circ\delta_B])$), which we explicit in the next remark for every pair $([f],[f\circ\delta_B])\in [\tilde {\mathbb K}, \mathbb K]_0\times[\tilde {\mathbb K}, \mathbb K]_0$, where for $[f]$ the Borsuk-Ulam Property does not hold in respect to $\delta_B$.		
		\begin{rem}\label{preimage OmegaB}
			Let $f:\mathbb K\rightarrow \mathbb K$ be a map such that its induced homomorphism $(f)_\#$ takes $a$ to $a_1^{2x}$ and $b$ to $a_1^{y}b_1^{z}$, where $x,y,z\in \mathbb Z$ and $z$ is odd.
			Consider the map $\Phi:\mathbb K\rightarrow B_2(\mathbb K)$ of type $(B)$ which induced homomorphism is given as follows
			\begin{align*}
				(\Phi)_\#:
				\alpha &\longmapsto  \hat \alpha=(b_2^{-1}a_2b_2a_2) a_1^{x}\sigma\\
				\beta &\longmapsto \hat \beta= (b_2^{-1}a_2b_2a_2)^l a_1^{y}b_1^{z},
			\end{align*}
			where $l\in \mathbb Z$.
			Using the relations of $B_2(\mathbb K)$ given in Proposition \ref{cap-braids:Presentation B_2(minha)(theo)} and Corollary \ref{relations sigma^2} it is straightforward to check that $\hat \alpha \hat \beta \hat \alpha= \hat \beta$ and therefore $\Phi$ is in fact a map.
			Note that $pr(b_2^{-1}a_2b_2a_2)=1$, where $pr$ is as in Remark \ref{elements in B2(K)}\eqref{pr}, therefore $m_1=m_2=n_1=n_2=0$ and
			Proposition \ref{PhiB0=>f1,f2} implies that $f_1$ is homotopic to $f$.
			From Theorem \ref{Phi map=>-f_1 BU} it follows that $f_2=f_1\circ \delta_B$, hence $\Omega_B$ maps $[\Phi]$ to $([f],[f\circ \delta_B])$.
			
			To complete the examples for all classes of maps which do not satisfy the Borsuk-Ulam Property in respect to $\delta_B$ (see Lemma \ref{K->K: not BU}), let
			$f:\mathbb K\rightarrow \mathbb K$ be a map such that its induced homomorphism $(f)_\#$ takes $a$ to $a_1^{2x+1}$ and $b$ to $a_1^{y}b_1^{z}$, where $x,y,z\in \mathbb Z$ and $z$ is odd.
			Then the $2$-valued non-split map, given in the proof of \cite[Proposition 32]{gonccalves2019borsuk}, with induced homomorphism
			\begin{align*}
				(\Phi)_\#:
				\alpha &\longmapsto  a_2^{-1}a_1^{x+1}\sigma\\
				\beta &\longmapsto a_2^{-1} a_1^{y}b_1^{z},
			\end{align*}			
			is such that for $\hat \Phi=(f_1,f_2)$, the lift of $\Phi\circ q$, it holds that $f$ is homotopic to $f_1$.
			Hence we have an example in the pre-image of $\Omega_B$ for each pair $([f],[f\circ \delta_B])\in [\tilde{\mathbb K},\mathbb K]_0\times[\tilde{\mathbb K},\mathbb K]_0$ such that $[f]$ does not satisfy the Borsuk Ulam Property in respect to $\delta_B$.
		\end{rem}
		For $\Omega_A$ there are also examples of $2$-valued non-split maps in the pre-image of each pair $([f],[f\circ \delta_A])\in [\tilde{\mathbb T},\mathbb K]_0\times[\tilde{\mathbb T},\mathbb K]_0$ such that $[f]$ does not satisfy the Borsuk Ulam Property in respect to $\delta_A$.
		These examples can be found in the proofs of Proposition 3.2. and 3.5. in \cite{gonccalves2022borsuk}, where the isomorphism between the presentation of $B_2(\mathbb K)$ used in \cite{gonccalves2022borsuk} (and \cite{gonccalves2019borsuk} for Remark \ref{preimage OmegaB}) and the one used in this paper (see Proposition \ref{cap-braids:Presentation B_2(minha)(theo)}) is given by
		\begin{multicols}{3}
			\noindent
			\begin{align*}
				(u;0, 0)&\mapsto \sigma ^{-2}a_2,\\
				(\vmathbb 1;1,0)&\mapsto \sigma^{-2}a_1,
			\end{align*}	
			
			\noindent
			\begin{align*}
				(v;0,0)&\mapsto b_2\sigma^2,\\
				(\vmathbb{1};0,1)&\mapsto b_1,
			\end{align*}	
			
			\noindent
			\begin{align*}	
				(B;0,0)&\mapsto \sigma^{-2},\\
				\sigma&\mapsto \sigma.
			\end{align*}
		\end{multicols}
		\noindent
		It is left to the reader to show that this is in fact an isomorphism.
		
		To have a classification of $2$-valued non-split maps on the Klein bottle could also help to study the Wecken property for $2$-valued maps (on the Klein bottle).
		Recall that a space satisfies the Wecken Property if for every self-map $f$ there is a map $\tilde f$ homotopic to $f$ with $N(f)$ fixed points.
		Analogously the space satisfies the Wecken Property for $2$-valued if we exchange the single valued map $f$ with a $2$-valued self-map $\phi$.
		For the Wecken Property for $2$-valued maps to hold, every map with Nielsen zero has to be homotopic to a fixed point free map.
		Taking Theorem \ref{Nielsennumber non-split} in account is easy to see that a $2$-valued non-split self-map on the Klein bottle has Nielsen number equal to zero if and only if one of the following is satisfied;
		\begin{itemize}
			\item $k_1=1, k_2=0$ and $s_2=1$,
			\item $k_1=1, k_2=1$ and $s_2-s_1=1$,
			\item $k_1=0, k_2=1$ and $n_2=2(1-s_2)$.
		\end{itemize}
		For the author it is currently unknown if a map with these conditions is homotopic to fixed point-free map.
		
		Another problem arising in this discussion, which is a current work in progress, is how to generalize Theorems \ref{Phi map=>-f_1 BU} and \ref{N(phi)=N(q,f1)} for $n\geq 2$, i.e. to study $n$-valued non-split maps, looking into the connection to the Borsuk-Ulam Property of their lift factors and searching for a method to compute the Nielsen fixed point number of $n$-valued non-split maps on orientable and non-orientable manifolds.


\bibliographystyle{plain} 
\bibliography{biblio} 

\end{document}